\PassOptionsToPackage{dvipsnames}{xcolor}

\documentclass[12pt]{amsart}

\usepackage{times}
\usepackage{pifont}

\usepackage{amsmath,amssymb,commath}
\usepackage[utf8]{inputenc}
\usepackage[T1]{fontenc}

\usepackage{hyperref}
\usepackage[capitalize,noabbrev]{cleveref}
\usepackage[backgroundcolor=yellow]{todonotes}

\usepackage{tikz}
\usetikzlibrary{matrix,arrows,calc}
\usetikzlibrary{decorations.markings}
\usetikzlibrary{patterns}
\usetikzlibrary{snakes}
\usetikzlibrary{shapes}

\makeatletter
\@ifclasswith{amsart}{externalize}%
{
	\usetikzlibrary{external} 
	\tikzexternalize[
	optimize=false,
	prefix=tikz/,
	mode=list and make]
	\makeatletter
	\renewcommand{\todo}[2][]{\tikzexternaldisable\@todo[#1]{#2}\tikzexternalenable}
	\makeatother
}%
{
	\newcommand{\tikzexternaldisable}{}
	\newcommand{\tikzexternalenable}{}
}
\makeatother

\usepackage{tikz-cd}
\usepackage{mathscinet}
\usepackage{enumitem}

\usepackage{array,multirow}

\makeatletter
\@ifclasswith{amsart}{biblatex}%
{
\usepackage[backend=bibtex,doi=true,isbn=false,url=false,style=alphabetic]{biblatex}

\addbibresource{shifted.bib}
\newcommand{\biblio}
{\printbibliography}
}
{
\newcommand{\biblio}%
{
\bibliographystyle{alpha}
\bibliography{Shifted.bib}}
}
\makeatother

\author[S.~Matsumoto]{Sho Matsumoto}
\address{
	Graduate School of Science and Engineering, Kagoshima University 
	1-21-35, Korimoto, Kagoshima, Japan 
}
\email{shom@sci.kagoshima-u.ac.jp}
 
\author[P.~\'Sniady]{Piotr \'Sniady}
\address{
Institute of Mathematics, Polish Academy of Sciences,
\mbox{ul.~\'Sniadec\-kich 8,} \linebreak 00-956 Warszawa, Poland
} 
\email{psniady@impan.pl}

\newcommand{\Z}{\mathbb{Z}}
\newcommand{\C}{\mathbb{C}}
\newcommand{\Q}{\mathbb{Q}}

\newcommand{\KO}{\mathbb{A}}

\newcommand{\F}{\mathcal{F}}
\newcommand{\G}{\mathcal{G}}

\newcommand{\Sym}[1]{\mathfrak{S}_{#1}}
\newcommand{\Spin}[1]{\widetilde{\mathfrak{S}}_{#1}}

\newcommand{\SGA}[1]{\C \Sym{#1}^-}

\newcommand{\ChSpin}{\Ch^{\mathrm{spin}}}
\newcommand{\Ch}{\mathrm{Ch}}
\newcommand{\ChD}{\widetilde{\Ch}}

\newcommand{\irrepSp}{\psi}
\newcommand{\double}{D^*}
\newcommand{\Dover}{D_{\mathrm{over}}}
\newcommand{\class}{\pi}
\newcommand{\partition}{I}
\newcommand{\sett}{S}
\newcommand{\setJ}{T}
\newcommand{\const}{C}

\newcommand{\CHIL}[2]{\chi^{#1}\left(#2\right)}

\newcommand{\XX}[2]{X^{#1}\left(#2\right)}

\newcommand{\PHI}[2]{\widetilde{\phi}^{#1}\left(#2\right)}
\newcommand{\PHIeasy}[2]{\phi^{#1}\left(#2\right)}

\newcommand{\proP}{\mathbf{p}}
\newcommand{\proQ}{\mathbf{q}}
\newcommand{\stimes}{\boxtimes}

\DeclareRobustCommand{\stirlingS}{\genfrac\{\}{0pt}{}}
\DeclareRobustCommand{\stirlingF}{\genfrac[]{0pt}{}}

\newcommand{\Part}{\mathcal{P}}
\newcommand{\SP}{\mathcal{SP}}
\newcommand{\OP}{\mathcal{OP}}

\DeclareMathOperator{\GL}{GL}
\DeclareMathOperator{\PGL}{PGL}

\DeclareMathOperator{\Tr}{Tr}

\DeclareMathOperator{\ttop}{top}
\DeclareMathOperator{\id}{id}

\theoremstyle{definition}
\newtheorem{definition}{Definition}[section]

\theoremstyle{theorem}
\newtheorem{theorem}[definition]{Theorem}
\newtheorem{proposition}[definition]{Proposition}
\newtheorem{corollary}[definition]{Corollary}
\newtheorem{lemma}[definition]{Lemma}

\theoremstyle{remark}
\newtheorem{example}[definition]{Example}
\newtheorem{remark}[definition]{Remark}

\begin{document}

\subjclass[2010]{%
	20C25  	% Projective representations and multipliers
	(Primary);
	20C30,  % Representations of finite symmetric groups
	05E05  	% Symmetric functions and generalizations
	(Secondary)}

\keywords{projective representations of the symmetric groups, linear
representations of the symmetric groups, asymptotic representation theory, Stanley character formula}

\title[Linear versus spin]{Linear versus spin: \\ 
representation theory\\ of the symmetric groups}

\begin{abstract}
We relate the \emph{linear} asymptotic representation theory of the symmetric
groups to its \emph{spin} counterpart. In particular, we give explicit formulas
which express the normalized irreducible \emph{spin} characters evaluated on a
strict partition $\xi$ with analogous normalized \emph{linear} characters
evaluated on the double partition $D(\xi)$. We also relate some natural
filtration on the usual (\emph{linear}) Kerov--Olshanski algebra of polynomial
functions on the set of Young diagrams with its \emph{spin} counterpart.
Finally, we give a spin counterpart to Stanley formula for the characters of the
symmetric groups.
\end{abstract}

\maketitle

\bigskip

\section{Introduction}

\emph{Spin} representation theory of the symmetric groups is a younger sibling
of the usual \emph{linear} representation theory. In the past, development of
the spin sibling took a path \emph{parallel} to that of the linear sibling
(often with a decades long delay):  the spin versions of the notions and
relationships were modelled after their linear analogues but there were somewhat
distinct and they had to be developed from scratch. The paths of the
developments of the siblings did not really \emph{cross}: there were not many
results which would directly link these two realms.

In the current paper we consider a setup in which there is a direct link between
the spin and the linear representation theory of the symmetric groups. In fact
we claim that \emph{the spin representation theory} (or at least some of its
aspects) \emph{can be derived from its linear counterpart}.

\medskip

Before we state our results
we shall recall some basic concepts of the spin representation theory of the
symmetric groups. For more details and bibliographic references refer to \cite{Wan2012,Kleshchev2005,Ivanov2004}.

\subsection{Spin representation theory of the symmetric groups}
\label{sec:projective-representations}

\subsubsection{Linear representations}
Recall that a \emph{linear} representation of a finite group $G$ 
is a group homomorphism $\irrepSp \colon G \to \GL(V)$
to the group of \emph{linear} transformations 
$\GL(V)$ 
of some finite-dimensional complex vector space~$V$.

\subsubsection{Projective representations}

A \emph{projective} representation of a finite group $G$ 
is a group homomorphism $\irrepSp \colon G \to \PGL(V)$
to the group of \emph{projective linear} transformations 
$\PGL(V)=\GL(V)/\C^{\times} $ of the projective space $P(V)$
for some finite-dimensional complex vector space $V$.
Equivalently, a projective representation can be viewed as a
map $\phi\colon G \to \GL(V)$ to the general linear group with the property that 
\[\irrepSp(x) \irrepSp(y)=c_{x,y}\ \irrepSp(xy) \]
holds true for all $x,y\in G$ for some non-zero scalar $c_{x,y}\in \C^{\times}$. 

\medskip

Each irreducible \emph{linear} representation $\irrepSp\colon G\to\GL(V)$  gives
rise to its projective version $\irrepSp\colon G\to\PGL(V)$. The irreducible
projective representations \emph{which cannot be obtained in this way} are
called \emph{irreducible spin representations} and are in the focus of the
current paper.

\subsubsection{Spin symmetric group and spin characters}
\label{sec:spin}

The \emph{spin group $\Spin{n}$} \cite{Schur1911}
is a double cover of the symmetric group:
\begin{equation} 
\label{eq:long-sequence}
1 \longrightarrow \Z_2=\{1,z\} \longrightarrow \Spin{n} \longrightarrow \Sym{n} \longrightarrow 1.
\end{equation}
More specifically, it is the 
group generated by $t_1,\dots,t_{n-1},z$
subject to the relations:
\begin{align*}
z^2   &= 1 ,\\
z t_i &= t_i z, & t_i^2&= z & \text{for $i\in [n-1]$}, \\
(t_i t_{i+1})^3 &= z  & & & \text{for $i\in [n-2]$}, \\
t_i t_j &=  z t_j t_i & & & \text{for $|i-j|\geq 2$};
\end{align*}	
we use the convention that $[k]=\{1,\dots,k\}$. Under the mapping
$\Spin{n}\to\Sym{n}$ the generators $t_1,\dots,t_{n-1}$ are mapped to the
Coxeter tranpositions $(1,2),\ (2,3),\dots,(n-1,n)\in\Sym{n}$. 

The main advantage of the spin group comes from the fact that any \emph{projective} representation
$\irrepSp\colon \Sym{n}\to \PGL(V)$ of the \emph{symmetric group} can be lifted
uniquely to a \emph{linear} representation
$\widetilde{\irrepSp}\colon\Spin{n}\to\GL(V)$ of the \emph{spin group} so that
the following diagram commutes:
\tikzexternaldisable
\[ \begin{tikzcd}
\Spin{n} \arrow[r, "\widetilde{\irrepSp}"] \arrow[d]
& \GL(V) \arrow[d] \\
\Sym{n} \arrow[r, "\irrepSp"]
& \PGL(V). \end{tikzcd}
\]
\tikzexternalenable
In this way the \emph{projective} representation theory of $\Sym{n}$ is
equivalent to the \emph{linear} representation theory of $\Spin{n}$ which allows
to speak about the characters.

\medskip

The
irreducible \emph{spin} representations of $\Sym{n}$ turn out to correspond to
irreducible \emph{linear} representations of the \emph{spin group algebra}
$\SGA{n}:= \C\Spin{n} / \langle z+1 \rangle$ which is the quotient of the group
algebra $\C\Spin{n}$ by the ideal generated by $(z+1)$. 

\subsubsection{Partitions} 

An \emph{integer partition} of $n\geq 0$ (often referred to as \emph{Young
	diagram} with $n$ boxes) is a finite sequence
$\lambda=(\lambda_1,\dots,\lambda_\ell)$ of positive integers (called
\emph{parts}) which is weakly decreasing $\lambda_1\geq \cdots\geq \lambda_\ell$
and such that $n=|\lambda|=\lambda_1+\cdots+\lambda_\ell$. The set of integer
partitions of a given integer $n\geq 0$ will be denoted by $\Part_n$, and the
set of \emph{all} integer partitions by $\Part=\bigsqcup_{n\geq 0} \Part_n$.
We sometimes write $\lambda \vdash n$ instead of $\lambda \in \Part_n$.

An integer partition is \emph{odd} if all of its parts are odd;
we denote by $\OP$ the set of odd partitions and
by $\OP_n$ the set of odd partitions of a given integer $n\geq 0$.

An integer partition $\xi=(\xi_1,\dots,\xi_\ell)$ is \emph{strict}
if it has no repeated entries; in other words the corresponding sequence  $\xi_1>\cdots>\xi_\ell$ is
\emph{strictly} decreasing. We denote by $\SP$
the set of strict partitions and by $\SP_n$ the set of strict partitions of a
given integer $n\geq 0$.

\medskip

We visualize partitions as Young diagrams drawn in the French convention and
strict partitions as \emph{shifted Young diagrams}, cf.~\cref{fig:strict}.

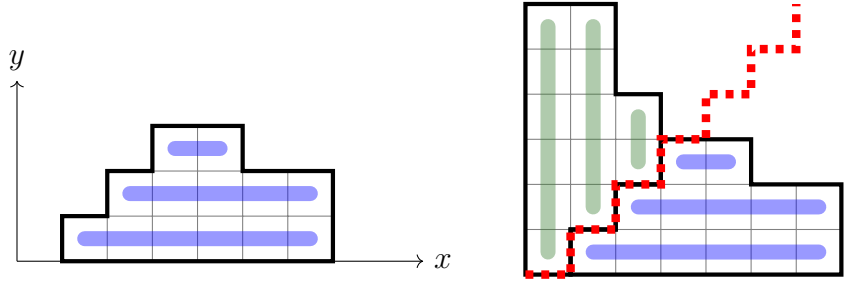
\begin{figure}[t]
	\centerline{
		\begin{tikzpicture}[xscale=0.6,yscale=0.6]
		\begin{scope}
		\clip (0,0) -- (0,1) -- (1,1) -- (1,2) -- (2,2) -- (2,3) -- (3,3) -- (4,3) -- (4,2) -- (6,2) -- (6,1) -- (6,0);
		\draw[gray] (0,0) grid (8,3);
		\end{scope}
		\draw[->] (-1,0) -- (8,0) node[anchor=west]{$x$};
		\draw[->] (-1,0) -- (-1,4) node[anchor=south]{$y$};	   
		\draw[ultra thick] (0,0) -- (0,1) -- (1,1) -- (1,2) -- (2,2) -- (2,3) -- (3,3) -- (4,3) -- (4,2) -- (6,2) -- (6,1) -- (6,0) -- cycle;
		\draw[line width=0.2cm, opacity=0.4,blue,line cap=round] (0.5,0.5) -- (5.5,0.5);
		\draw[line width=0.2cm, opacity=0.4,blue,line cap=round] (1.5,1.5) -- (5.5,1.5);
		\draw[line width=0.2cm, opacity=0.4,blue,line cap=round] (2.5,2.5) -- (3.5,2.5);	   
		\end{tikzpicture}
		\hspace{3ex}
		\begin{tikzpicture}[xscale=0.6,yscale=0.6]
		\begin{scope}[xshift=1cm]
		\clip (0,0) -- (0,1) -- (1,1) -- (1,2) -- (2,2) -- (2,3) -- (3,3) -- (4,3) -- (4,2) -- (6,2) -- (6,1) -- (6,0);
		\draw[gray] (0,0) grid (8,3);
		\end{scope}
		\begin{scope}[rotate=90,yscale=-1]
		\clip (0,0) -- (0,1) -- (1,1) -- (1,2) -- (2,2) -- (2,3) -- (3,3) -- (4,3) -- (4,2) -- (6,2) -- (6,1) -- (6,0);
		\draw[gray] (0,0) grid (8,3);
		\end{scope}    
		\begin{scope}[xshift=1cm]
		\draw[ultra thick] (0,0) -- (0,1) -- (1,1) -- (1,2) -- (2,2) -- (2,3) -- (3,3) -- (4,3) -- (4,2) -- (6,2) -- (6,1) -- (6,0) -- cycle;
		\draw[line width=0.2cm, opacity=0.4,blue,line cap=round] (0.5,0.5) -- (5.5,0.5);
		\draw[line width=0.2cm, opacity=0.4,blue,line cap=round] (1.5,1.5) -- (5.5,1.5);
		\draw[line width=0.2cm, opacity=0.4,blue,line cap=round] (2.5,2.5) -- (3.5,2.5);	 
		\end{scope}
		\begin{scope}[rotate=90,yscale=-1]
		\draw[ultra thick] (0,0) -- (0,1) -- (1,1) -- (1,2) -- (2,2) -- (2,3) -- (3,3) -- (4,3) -- (4,2) -- (6,2) -- (6,1) -- (6,0) -- cycle;
		\draw[line width=0.2cm, opacity=0.4,OliveGreen,line cap=round] (0.5,0.5) -- (5.5,0.5);
		\draw[line width=0.2cm, opacity=0.4,OliveGreen,line cap=round] (1.5,1.5) -- (5.5,1.5);
		\draw[line width=0.2cm, opacity=0.4,OliveGreen,line cap=round] (2.5,2.5) -- (3.5,2.5);	
		\end{scope}    
		\draw[line width=3.2pt, dashed,red] (0,0) -- (1,0) -- (1,1) -- (2,1) -- (2,2) -- (3,2) -- (3,3) -- (4,3) -- (4,4) -- (5,4) -- (5,5) -- (6,5) -- (6,6);
		\end{tikzpicture}
	}
	
	\caption{ Strict partition $\xi=(6,5,2)$ shown as a \emph{shifted Young
			diagram} and its double $D(\xi)=(7,7,5,3,2,2)$. } 
	\label{fig:strict}
	\label{fig:double}
\end{figure}

\subsubsection{Spin characters} 
\label{sec:intro-spin}

Schur \cite{Schur1911} proved that, roughly
speaking\footnote{\label{note1}For a precise statement see \cref{sec:spin-characters}.}, 
the conjugacy classes of $\Spin{n}$ which are non-trivial from the
viewpoint of the spin character theory are indexed by the elements of the set
$\OP_n$, i.e.~by the odd partitions of $n$.

Schur also proved that, roughly speaking\textsuperscript{\ref{note1}}, the
irreducible spin representations of $\Spin{n}$ are indexed by the elements of
the set $\SP_n$, i.e.~by strict partitions of $n$

For an odd partition $\pi\in\OP_n$ (which corresponds to a conjugacy class of
$\Spin{n}$) and a strict partition $\xi\in\SP_n$ (which corresponds to its
irreducible spin representation) we denote by $\PHIeasy{\xi}{\pi}$ the corresponding
\emph{spin character} (for some fine details related to this definition we refer
to \cref{sec:spin-characters}). 

\subsection{Normalized characters}
The usual way of viewing the linear characters $\CHIL{\lambda}{\pi}$ of the
symmetric group $\Sym{n}$ is to fix the irreducible representation $\lambda$ and
to consider the character as a function of the conjugacy class $\class$. 
The \emph{dual approach}, initiated by Kerov and Olshanski \cite{Kerov1994},
suggests to do the opposite: \emph{fix the conjugacy class $\class$ and to view
	the character as a function of the irreducible representation $\lambda$}. 
In
order for this approach to be successful one has to choose the most convenient
normalization constants which we review in the following.
\medskip

Following Kerov and Olshanski \cite{Kerov1994}, 
for a fixed integer partition $\class$ the
corresponding \emph{normalized linear character on the conjugacy class $\class$}
is the function on the set of \emph{all} Young diagrams given by
\[ \Ch_\class(\lambda):=\begin{cases}
n^{\downarrow k} 
\ \frac{ \CHIL{\lambda}{\class\cup 1^{n-k}} }{ \CHIL{\lambda}{1^{n}} } & 
\text{if } n\geq k, \\
0 & \text{otherwise,}
\end{cases}\]
where $n=|\lambda|$ and $k=|\class|$ and $n^{\downarrow k}=n (n-1) \cdots
(n-k+1)$ denotes the falling power. Above, for partitions $\lambda,\sigma\vdash
n$ we denote by $\CHIL{\lambda}{\sigma}$ the irreducible linear character of the
symmetric group which corresponds to the Young diagram $\lambda$, evaluated on
any permutation with the cycle decomposition given by $\sigma$.

\medskip

Following Ivanov \cite{Ivanov2004,Ivanov2006}, for a fixed odd partition
$\class\in\OP$ the corresponding \emph{normalized spin character} is a function
on the set of \emph{all} strict partitions given by
\begin{equation}
\label{eq:projective-normalized}
\ChSpin_\class(\xi):=\begin{cases}
n^{\downarrow k}\ 2^{\frac{k-\ell(\class)}{2}}
\ \frac{ \PHIeasy{\xi}{\class\cup 1^{n-k}}}{ \PHIeasy{\xi}{1^{n}}}
& \text{if } n\geq k, \\
0 & \text{otherwise,}
\end{cases}
\end{equation}
where $n=|\xi|$,  $k=|\class|$, 
and $\ell(\pi)$ denotes the number of parts of $\pi$.

\bigskip

\textbf{Our goal in this paper is to find relationships between the family
$(\Ch_\class)$ of linear characters and its spin counterpart $(\ChSpin_\class)$. This
quite general goal can take various more specific incarnations which we will
discuss in the following.}

\subsection{The main result 1. Characters: linear vs spin}

Firstly, we might look for some explicit algebraic formulas which express the
linear characters $(\Ch_\class)$ in terms of the spin characters $(\ChSpin_\class)$
and vice versa.

The first step in this direction is to relate the set $\SP$ of strict partitions
to the set $\Part$ of partitions which are, respectively, the domain of the linear
characters $(\Ch_\class)$ and the domain of the spin characters $(\ChSpin_\class)$.
More specifically, for a strict partition $\xi\in\SP_n$ we consider its
\emph{double} $D(\xi)\in \Part_{2n}$. Graphically, $D(\xi)$ corresponds to a
Young diagram obtained by arranging the shifted Young diagram $\xi$ and its
`transpose' so that they nicely fit along the `diagonal', cf.~\cref{fig:double},
see also \cite[Example I.1.9(a)]{Macdonald1995}.

\medskip

In \cref{sec:spin-in-terms-of-linear} we will show several formulas which relate
linear and spin characters. The following is a particularly simple special case
which nevertheless gives a flavour of the general case.

\begin{theorem}[The main result for characters, special case]
\label{thm:main-special}
\ \\ For any odd integer $k\geq 1$ and any strict partition $\xi\in\SP$
	\[ \Ch_{k}\big( D(\xi) \big) = 2\ \ChSpin_k(\xi). \]
\end{theorem}

We will prove this result (in wider generality) in \cref{sec:proof}.
The results of this type reduce the problem of calculating of the spin
characters to its linear counterpart for which often more tools are available
\cite{Rattan'Sniady2008,FeraySniady2011a,GouldenRattan2007,Biane2005/07,PetrulloSenato2011,DolegaFeraySniady2008}. 

\subsection{The main result 2. Spin version of Stanley character formula}

\subsubsection{Colorings} 
\label{sec:colorings}

Let $\sigma_1,\sigma_2\in\Sym{k}$ be permutations and
let $\lambda\in\Part$ be a Young diagram. Following \cite{FeraySniady2011a}, we
say that $(f_1,f_2)$ is a \emph{coloring} of $(\sigma_1,\sigma_2)$ which is
compatible with $\lambda$ if the following two conditions are fulfilled:
\begin{itemize}
\item $f_i\colon C(\sigma_i) \to \Z_+$ is a function on the set of cycles of
$\sigma_i$ for each $i\in\{1,2\}$; we view the values of $f_1$ as columns of
$\lambda$ and the values of $f_2$ as rows;

\item whenever $c_1\in C(\sigma_1)$ and $c_2\in C(\sigma_2)$ are cycles which are not
disjoint ($c_1\cap c_2\neq \emptyset$), the box with Cartesian coordinates
$\big( f_1(c_1), f_2(c_2) \big)$ belongs to $\lambda$.
\end{itemize} 
We denote by $N_{\sigma_1,\sigma_2}(\lambda)$ the number of such colorings of
$(\sigma_1,\sigma_2)$ which are compatible with $\lambda$.

\begin{figure}
	\centerline{
		\begin{tikzpicture}[scale=1]
		\draw [->] (0,0) -- (5,0) node[anchor=west] {$x$};
		\draw [->] (0,0) -- (0,3) node[anchor=west] {$y$};
		\begin{scope}
		\clip (0,0) -- (3,0) -- (3,1) -- (1,1) -- (1,2) -- (0,2);
		\draw (0,0) grid (2,2);
		\end{scope}
		\draw[ultra thick] (0,0) -- (3,0) -- (3,1) -- (1,1) -- (1,2) -- (0,2) -- cycle;
		\draw[line width=7pt, opacity=0.3, draw=red] (0.5,-0.5) node[anchor=north,opacity=1, text=black]{$f_1(V)$} -- (0.5,2.5); 
		\draw[line width=7pt, opacity=0.3, draw=blue] (2.5,-0.5) node[anchor=north,opacity=1, text=black]{$f_1(W)$} -- (2.5,1.5); 
		\draw[line width=7pt, opacity=0.3, draw=OliveGreen] (-0.5,0.5) node[anchor=east,opacity=1, text=black]{$f_2(\Pi)$} -- (3.5,0.5); 
		\draw[line width=7pt, opacity=0.3, draw=black] (-0.5,1.5) node[anchor=east,opacity=1, text=black]{$f_2(\Sigma)$} -- (1.5,1.5); 
		\end{tikzpicture}
	} 
	\caption{Graphical representation of the coloring \eqref{eq:coloring} of the
		permutations \eqref{eq:example}  which is compatible with the Young diagram
		$\lambda=(3,1)$.} \label{fig:embed}
\end{figure}
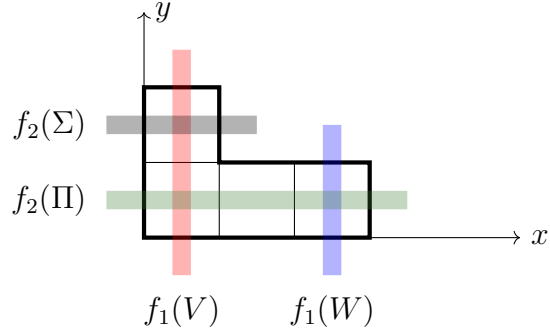

\begin{example}
	Let
	\begin{equation}
	\label{eq:example} 
	\sigma_1=\underbrace{(1,5,4,2)}_{\textcolor{red}{V}}\underbrace{(3)}_{\textcolor{blue}{W}},\qquad  \sigma_2 = \underbrace{(2,3,5)}_{\textcolor{OliveGreen}{\Pi}}\underbrace{(1,4)}_{\Sigma}.
	\end{equation}
	There are three pairs of cycles $(\sigma_1,\sigma_2)\in C(\sigma_1)\times
	C(\sigma_2)$ with the property that $\sigma_1$ and $\sigma_2$ are not disjoint,
	namely $(V,\Pi), (V,\Sigma), (W,\Pi)$.
	It is now easy to check graphically (cf.~\cref{fig:embed}) that $(f_1,f_2)$ given by 
	\begin{equation} 
	\label{eq:coloring}
	f_1(V)=1, \quad f_1(W)=3, \quad f_2(\Pi)=1, \quad f_2(\Sigma)=2
	\end{equation}
	is indeed an example of a coloring compatible with $\lambda=(3,1)$.

	By considering four possible choices for the values of $f_2$ and counting the
	choices for the values of $f_1$ one can verify that \[ N_{\sigma_1,\sigma_2}(\lambda)=
	3^2 +  3 + 1 + 1 =14\qquad \text{for }\lambda=(3,1).\]
\end{example}

\subsubsection{Linear Stanley character formula} 

We will identify a given integer partition $\pi=(\pi_1,\dots,\pi_\ell)\vdash k$
with an arbitrary permutation $\pi\in\Sym{k}$ with the corresponding cycle
structure. For example, we may take
\[ \pi=(1,2,\dots,\pi_1)(\pi_1+1,\pi_1+2,\dots,\pi_1+\pi_2) \cdots\in\Sym{k}.\]

The following result turned out in the past to be a convenient tool for studying
asymptotic \cite{FeraySniady2011a} and enumerative problems
\cite{DolegaFeraySniady2008} of the (linear) representation theory of the
symmetric groups.

\begin{theorem}[\cite{FeraySniady2011a}]
	\label{thm:stanley-linear}
	For any partition $\pi\in\Part_k$ and any Young diagram $\lambda\in\Part$
	\begin{equation}  
	\label{eq:stanley-linear} 
	\Ch_\pi(\lambda) = \sum_{\substack{\sigma_1,\sigma_2\in\Sym{k} \\ \sigma_1 \sigma_2=\pi }}  
	(-1)^{\sigma_1} N_{\sigma_1,\sigma_2} (\lambda). 
	\end{equation}
\end{theorem}

\subsubsection{Spin Stanley character formula}

For $\sigma_1,\sigma_2\in\Sym{k}$ we denote by \mbox{$|\sigma_1\vee \sigma_2|$} the
number of %set-partition of $[k]$ equal to
orbits in the set $[k]=\{1,\dots,k\}$ under the action of the group $\langle
\sigma_1,\sigma_2\rangle$ generated by $\sigma_1,\sigma_2$. Equivalently, it is the number of parts of the
set-partition $\sigma_1 \vee \sigma_2$ (which is the minimal set-partition with
respect to the refinement order which is greater than each of the two
set-partitions which correspond to the cycles of $\sigma_1$ and $\sigma_2$).

One of the main results of the current paper is the following spin analogue of
\cref{thm:stanley-linear}.

\begin{theorem}[The main result: spin Stanley character formula]
	\label{thm:spin-Stanley}
	For any odd partition $\pi\in\OP_k$ and any strict partition $\xi\in\SP$
	\begin{equation}
	\label{eq:spin-stanley}
	  \ChSpin_\pi(\xi) = \sum_{\substack{\sigma_1,\sigma_2\in\Sym{k} \\ \sigma_1 \sigma_2=\pi }} 
	  \frac{1}{2^{|\sigma_1\vee \sigma_2|}}\ (-1)^{\sigma_1}\ N_{\sigma_1,\sigma_2}\big( D(\xi) \big). 
	  \end{equation}
\end{theorem}
The proof is postponed to \cref{sec:proof-of-stanley}. 

\subsubsection{Application: combinatorics of spin Kerov polynomials and Stanley polynomials}

In the context of the usual (linear) asymptotic representation theory of the
symmetric groups it is convenient to parametrize the set of Young diagrams by
\emph{free cumulants} \cite{Biane1998}. Each linear character $\Ch_\pi$ can be
then expressed uniquely as a polynomial (called \emph{Kerov character
	polynomial}) in these free cumulants. Kerov polynomials turned out to have a
very rich structure from the viewpoint of enumerative and algebraic
combinatorics
\cite{Biane2003,GouldenRattan2007,Biane2005/07,PetrulloSenato2011,DolegaFeraySniady2008,Sniady2013}. 
Combinatorics of Kerov polynomials is intimately related to
\emph{multirectangular coordinates} of Young diagrams and the corresponding
\emph{Stanley character polynomials}, see \cref{sec:multirectangular}.

Recently De Stavola \cite{DeStavolaThesis} and the first-named author
\cite{Matsumoto2018} initiated investigation of spin counterparts of these
notions. In a forthcoming publication \cite{MatsumotoSniady2018c} we will
present applications of \cref{thm:spin-Stanley} to combinatorics of spin Kerov
polynomials and spin Stanley character polynomials.

\subsubsection{Application: bounds on spin characters}

The following character bound is a spin version of an analogous result for the
linear characters of the symmetric group \cite{FeraySniady2011a}. It is a direct
application of \cref{thm:spin-Stanley} and its proof follows the same line as
its linear counterpart \cite{FeraySniady2011a}.

\begin{corollary}
	\label{cor:Annals}
	There exists a universal constant $a>0$ with the property that for any integer
	$n\geq 1$, any strict partition $\xi=(\xi_1,\xi_2 \dots)\in\SP_n$, and any odd partition
	$\pi\in\OP_n$
	\[
	\left| \frac{\PHIeasy{\xi}{\pi}}{\PHIeasy{\xi}{1^n}} \right| 
< \left[a \max\left( \frac{\xi_1}{n} , \frac{n-\ell(\pi)}{n} \right) \right]^{n-\ell(\pi)}. \]
\end{corollary}	

Several asymptotic results about (random) Young diagrams and tableaux which use
the inequality from \cite{FeraySniady2011a} can be generalized in a rather
straightforward way to (random) \emph{shifted} Young diagrams and \emph{shifted}
tableaux thanks to \cref{cor:Annals}. A good example is provided by the results
of Dousse and F\'eray about the asymptotics of the number of skew standard Young
tableaux of prescribed shape \cite{Dousse2017} which can be generalized in this
way to asymptotics of the number of skew \emph{shifted} standard Young tableaux.

\subsubsection{Open problem: towards irreducible representations of spin groups} 

The proof of
the linear Stanley formula \eqref{eq:stanley-linear} presented in
\cite{FeraySniady2011a} was found in the following way. 
We attempted to
reverse-engineer the right-hand side of \eqref{eq:stanley-linear} and to find
\begin{itemize}
\item some natural vector space $V$ with the basis indexed by combinatorial objects;
the space $V$ should be a representation of the symmetric group $\Sym{n}$ with
$n:=|\lambda|$, and 
\item  a projection $\Pi\colon V\to V$ such that $\Pi$ commutes
with the action of $\Sym{n}$ and such that it image $\Pi V$ is an irreducible
representation of $\Sym{n}$ which corresponds to the specified Young diagram
$\lambda$,
\end{itemize}
in such a way that the corresponding character of $\Pi V$ would coincide with
the right-hand side of \eqref{eq:stanley-linear}.

Our attempt was successful: one could consider a vector space $V$ with the basis
indexed by fillings of the boxes of $\lambda$ with the numbers from $[n]$ with the
action of $\Sym{n}$ given by pointwise relabelling of the values in the boxes.
The projection $\Pi$ turned out to be the Young symmetrizer with the action
given by shuffling of the boxes in the rows and columns of $\lambda$. The
resulting representation $\Pi V$ clearly coincides with the Specht module, which
concluded the proof in \cite{FeraySniady2011a}.

\medskip

The structure of the right-hand side of \eqref{eq:spin-stanley} 
(as well as the right-hand side of \eqref{eq:spin-stanley-nonoriented}) 
might be an
indication that an analogous reverse-engineering process could be applied to the
spin case. The result would be a very explicit construction of the irreducible
spin representations which would be an alternative to the somewhat complicated approach of Nazarov \cite{Nazarov1990}.

\subsection{Stanley formulas and enumeration of maps} %\todo{work in progress
Some readers may have aesthetic objections against Equation \eqref{eq:spin-stanley} related to the somewhat ugly factor
$\frac{1}{2^{|\sigma_1\vee \sigma_2|}}$ on the right-hand side. Our remedy to
this issue is \cref{thm:maps} which avoids such factors. This result is just a
reformulation of \cref{thm:spin-Stanley} in the language of \emph{non-oriented
	maps which are orientable}. We start by introducing the notations which are
necessary to state this result.

\subsubsection{Maps}
\label{sec:maps}
We recall that a \emph{map} \cite{Lando2004} is defined as a graph $\mathcal{G}$ drawn on
a surface $\Sigma$ without boundary in such a way that each connected component of
$\Sigma\setminus \mathcal{G}$ (called \emph{face}) is
homeomorphic to an open disc.
In the literature one often adds the requirement that the surface $\Sigma$ is connected; we will not impose such a restriction.

Each map which we will consider is a \emph{bicolored map} which means that the
set of its vertices $\mathcal{V}=\mathcal{V}_\circ\sqcup \mathcal{V}_\bullet$ is
canonically decomposed into two classes (usually called white and black
vertices), with edges connecting only vertices of the opposite colors. Let
$F_1,\dots,F_\ell$ be the faces of the map; for the above reason of
bicoloration, each face $F_i$ is a polygon which consists of an even number
(say, $2\pi_i$) of edges, with neighbouring vertices painted in the alternating
colours. The integer partition $\pi=(\pi_1,\dots,\pi_\ell)$ is called the
\emph{face-type} of the map, see \cref{fig:map-projective-plane}.

With this in mind, one can alternatively view a bicolored map as a collection
$\mathcal{P}$ of polygons with labelled edges and with the vertices painted in
the alternating colours (see \cref{fig:polygons}), together with the information which pairs of the edges
should be glued together. Note that at first moment it might appear that there
are two ways to glue any given pair of edges (these two ways differ by a twist a
la M\"obius), nevertheless the bicoloration of the vertices makes the choice
unique.

\subsubsection{Non-oriented maps}
\label{sec:sum-nonoriented}

The above discussion motivates the following definition. \emph{Summation over non-oriented maps
	with a specified face-type $\pi$} \cite[Section 3.4]{DolegaFeraySniady2014}
should be understood as follows: we start by fixing an appropriate collection $\mathcal{P}$ of
polygons with labelled edges, bicolored vertices, and with face-type $\pi$.
Then we consider all perfect matchings on the set of
all edges, and for each such a perfect matching we glue the corresponding pairs
of the edges. We sum over the resulting map.

\newcommand{\faceA}{red!50}
\newcommand{\faceB}{blue}
\newcommand{\faceAfill}{red!10}
\newcommand{\faceBfill}{blue!20}

\begin{figure}[t]
	\begin{tikzpicture}[
	yscale=-0.5,xscale=0.5,
	black/.style={circle,thick,draw=black,fill=white,inner sep=4pt},
	white/.style={circle,draw=black,fill=black,inner sep=4pt},
	connection/.style={draw=black,black,auto}
	]
	\small
	
	\fill[pattern color=\faceAfill,pattern=north west lines]  (0*36:5) -- (1*36:5) -- (2*36:5) -- (3*36:5) -- (4*36:5) -- (5*36:5) -- (6*36:5) -- (7*36:5) -- (8*36:5) -- (9*36:5);
	
	\draw (0*36:5)  node (b1)     [black] {};
	\draw (1*36:5)  node (b2)     [white] {};
	\draw (2*36:5)  node (b3)     [black] {};
	\draw (3*36:5)  node (b4)     [white] {};
	\draw (4*36:5)  node (b5)     [black] {};
	\draw (5*36:5)  node (b6)     [white] {};
	\draw (6*36:5)  node (b7)     [black] {};
	\draw (7*36:5)  node (b8)     [white] {};
	\draw (8*36:5)  node (b9)     [black] {};
	\draw (9*36:5)  node (b10)    [white] {};
	
	\draw (0.5*36:4.1)  node {$1$};
	\draw (1.5*36:4.1)  node {$2$};
	\draw (2.5*36:4.1)  node {$3$};
	\draw (3.5*36:4.1)  node {$4$};
	\draw (4.5*36:4.1)  node {$5$};
	\draw (5.5*36:4.1)  node {$6$};
	\draw (6.5*36:4.1)  node {$7$};
	\draw (7.5*36:4.1)  node {$8$};
	\draw (8.5*36:4.1)  node {$9$};
	\draw (9.5*36:4.1)  node {$10$};

	\draw[connection,decoration={
		markings,
		mark=at position 0.5  with {\arrow[red,line width=1.5mm]{right to}}},
	postaction={decorate}]  
	(b1) to (b2);
    \draw[red, line width=0.5mm,->] (b1)+(-2,0) to (b1);

	\draw[connection] (b2) to  (b3);
	\draw[connection] (b3) to  (b4);
	\draw[connection] (b4) to  (b5);
	\draw[connection] (b5) to  (b6);
	\draw[connection] (b6) to  (b7);
	\draw[connection] (b7) to  (b8);
	\draw[connection] (b8) to  (b9);
	\draw[connection] (b9) to  (b10);
	\draw[connection] (b10)to (b1);
	
	\begin{scope}[shift={(10,0)},scale=0.5]
	\fill[\faceBfill]  (0:5) -- (90:5) -- (180:5) -- (270:5);
	
	\draw (90*0:5)  node (b1)     [black] {};
	\draw (90*1:5)  node (b2)     [white] {};
	\draw (90*2:5)  node (b3)     [black] {};
	\draw (90*3:5)  node (b4)     [white] {};
	
	\draw (90*0.5:2.2) node {$A$};
	\draw (90*1.5:2.2) node {$B$};
	\draw (90*2.5:2.2) node {$C$};
	\draw (90*3.5:2.2) node {$D$};
	
	\draw[connection,decoration={
	markings,
	mark=at position 0.5  with {\arrow[blue,line width=1.5mm]{right to}}},
postaction={decorate}]  
(b1) to (b2);

    \draw[blue, line width=0.5mm,->] (b1)+(-3,0) to (b1); 
	
	\draw[connection] (b2) to  (b3);
	\draw[connection] (b3) to  (b4);
	\draw[connection] (b4) to  (b1);
	\end{scope}
	\end{tikzpicture}
	
	\caption{Example of a collection of polygons with face-type $(5,2)$. Arrows
	(pointing at white vertices between the edges $1$ and $10$, as well as between
	$A$ and $D$) indicate the roots; harpoons (at the edges $1$ and $A$)
	indicate the directions of rotation in each
	polygon, see \cref{sec:non-oriented-revisited}.} \label{fig:polygons}
	
	\vspace{9ex}
	
	\begin{tikzpicture}[
yscale=-0.6,xscale=0.6,
black/.style={circle,thick,draw=black,fill=white,inner sep=4pt},
white/.style={circle,draw=black,fill=black,inner sep=4pt},
connection/.style={draw=black,black,auto}
]
\scriptsize

\begin{scope}
\clip (0,0) rectangle (10,10);

\fill[pattern color=\faceAfill,pattern=north west lines] (0,0) rectangle (10,10);
\fill[\faceBfill] (5,5) rectangle (8,8);

\draw (2,7)  node (b1)     [black] {};
\draw (12,3) node (b1prim) [black] {};
\draw (3,2)  node (w1) [white] {};
\draw (-7,8) node (w1prim) [white] {};

\draw (5,5)  node (AA)     [black] {};
\draw (8,5)  node (BA)     [white] {};
\draw (5,8)  node (AB)     [white] {};
\draw (8,8)  node (BB)     [black] {};

\draw[connection]         (w1)      to node [swap,pos=0.6] {$4$} node [pos=0.6] {$9$} (AA);
\draw[connection]         (AA)      to node [swap] {$5$} node {$D$} (AB);
\draw[connection]         (AB)      to node [swap] {$6$} node {$C$} (BB);
\draw[connection]         (BB)      to node [swap] {$7$} node {$B$} (BA);
\draw[connection,decoration={
	markings,
	mark=at position 0.5  with {\arrow[blue,line width=1.5mm]{right to}}},
postaction={decorate}]         (AA)      to node {$8$} node [swap]  {$A$} (BA);
    \draw[blue, line width=0.5mm,->] (AA)+(1,1) to (AA); 

\draw[connection]         (w1)      to node [swap,pos=0.5] {$10$} node [pos=0.425] {$2$} (b1prim);
\draw[connection]         (w1prim)  to node [swap,pos=0.87] {$2$} node [pos=0.95] {$10$} (b1);

\draw[connection,decoration={
	markings,
	mark=at position 0.5  with {\arrow[red,line width=1.5mm]{left to}}},
postaction={decorate}]         (b1)      to node [pos=0.575] {$1$}  node [swap,pos=0.5] {$3$} (w1);

    \draw[red, line width=0.5mm,->] (b1)+(-1,-1) to (b1); 

\end{scope}

\draw[very thick,decoration={
	markings,
	mark=at position 0.666  with {\arrow{>}}},
postaction={decorate}]  
(0,0) -- (10,0);

\draw[very thick,decoration={
	markings,
	mark=at position 0.666  with {\arrow{>}}},
postaction={decorate}]  
(10,10) -- (0,10);

\draw[very thick,decoration={
	markings,
	mark=at position 0.666  with {\arrow{>>}}},
postaction={decorate}]  
(10,0) -- (10,10);

\draw[very thick,decoration={
	markings,
	mark=at position 0.666  with {\arrow{>>}}},
postaction={decorate}]  
(0,10) -- (0,0);

\end{tikzpicture}
	\caption{Example of a \emph{non-oriented map} with face-type $(5,2)$ drawn on the projective plane.
		The left side of the square should be glued with a twist to the right side, 
		as well as bottom to top (also with a twist), as indicated by arrows. 
		This map has been obtained by gluing the edges of the polygons from Figure~\ref{fig:polygons}. }
	\label{fig:map-projective-plane}
	
\end{figure}

\subsubsection{Colorings revisited} 

The notion of \emph{number of colorings} $N_{\sigma_1,\sigma_2}(\lambda)$ was
introduced in \cref{sec:colorings} in the context of a pair of permutations. We review how to generalize it to maps and, more generally, bicolored graphs.

Let $\mathcal{G}$ be a bicolored graph with the vertex set $\mathcal{V}=\mathcal{V}_\circ\sqcup \mathcal{V}_\bullet$ and
let $\lambda\in\Part$ be a Young diagram. 
We say that $(f_1,f_2)$ is a \emph{coloring} of $\mathcal{G}$ which is
compatible with $\lambda$ if the following two conditions are fulfilled:
\begin{itemize}
	\item $f_1\colon \mathcal{V}_\circ \to \Z_+$ is a function on the set of white
vertices while $f_2\colon \mathcal{V}_\bullet \to \Z_+$ is a function on the
set of black vertices;
	
	\item whenever vertices $v_1\in \mathcal{V}_\circ$ and $v_2\in \mathcal{V}_\bullet$ are connected by an edge,  the box with Cartesian coordinates
	$\big( f_1(c_1), f_2(c_2) \big)$ belongs to $\lambda$.
\end{itemize} 
We denote by $N_{\mathcal{G}}(\lambda)$ the number of such colorings of
$\mathcal{G}$ which are compatible with $\lambda$.

\begin{remark}
In \cref{sec:labels} we shall explain how to associate a map $\mathcal{G}$ to a pair
$(\sigma_1,\sigma_2)$ of permutations; the above definition of $N_{\mathcal{G}}(\lambda)$ is compatible with this correspondence
in the sense that
\[ N_{\sigma_1,\sigma_2}(\lambda)= N_{\mathcal{G}}(\lambda),\]
where the left-hand side should be understood in the sense of \cref{sec:labels}.
\end{remark}

\subsubsection{Spin Stanley formula and non-oriented, orientable maps}

The following is a reformulation of \cref{thm:spin-Stanley}.
\begin{theorem}
	\label{thm:maps}
	For any odd partition $\pi\in\OP_k$ and any strict partition $\xi\in\SP$
	\begin{equation} 
	\label{eq:spin-stanley-nonoriented}  2^{\ell(\pi)}\ \ChSpin_\pi (\zeta) 
	= \sum_M (-1)^{|\pi|- |\mathcal{V}_\circ(M)|}\ N_M\big( D(\zeta) \big)  
	\end{equation}
	where 
	the sum runs over non-oriented, orientable bicolored maps with the face-type $\pi$, as in \cref{sec:sum-nonoriented}.
\end{theorem}

The proof is postponed to \cref{sec:proof-of-maps}.

\renewcommand{\arraystretch}{1.5}

	\begin{table}
		\begin{tabular}{|m{4.1cm}|m{3.3cm}|m{3.2cm}|}
			\hline
			\textsc{type of characters}     & \textsc{type of maps}    & \textsc{references}            
			\\ \hline \hline
			\parbox[t]{4.1cm}{linear characters \\ of the symmetric groups \\ }    & oriented maps   & 	\parbox[t]{3.2cm}{\cite[Sections 6.4--6.5]{Sniady2013}, \\   \cite[Section 8]{Sniady2016}  \\     }                \\ \hline
			\parbox[t]{4.1cm}{spin characters \\ of the symmetric groups \\}       & \parbox[t]{3.3cm}{non-oriented maps \\ which are \\ orientable \\}            &  \parbox[t]{3.2cm}{current paper,\\ \cref{thm:maps}} 
			\\ \hline
			zonal characters        & \multirow{2}{3.5cm}{non-oriented maps}  &    \multirow{2}{*}{\cite[Section 5]{Feray2011}}  \\ \cline{1-1}
			\parbox[t]{4.1cm}{symplectic zonal \\ characters \\ }   &                                     &                       \\ \hline
			\multirow{2}{5cm}{top-degree \\ of Jack characters} & \parbox[t]{3.3cm}{oriented, \\ connected maps \\ with arbitrary \\  face structure \\ }      & \parbox[t]{3.3cm}{\cite[Sections 1.4.2--1.4.3]{SniadyTopDegree}}   \\ \cline{2-3} 
			     & \parbox[t]{3.3cm}{non-oriented maps \\ with prescribed \\ face structure \\} & \cite[Section 4]{Czyzewska-Jankowska2017} \\ \hline
\parbox[t]{4.1cm}{Jack characters \\ in the generic case} & 
\parbox[t]{3.3cm}{conjecturally: \\ non-oriented maps \\
	(counted with some unknown weight) \\ } & \parbox[t]{3.3cm}{partial results:\\ \cite{DolegaFeraySniady2014}} \\ \hline
		\end{tabular}
	\vspace{2ex}
	
	\caption{Classes of functions on the set of (shifted) Young diagrams for which
	some version of Stanley character formula is known 
	(or conjectured in the case of Jack characters) and the corresponding class
	of maps over which the summation is performed.}
\label{tab:table}
	\end{table}

\subsubsection{Stanley formulas and maps}

\cref{tab:table} summarizes some known and some hypothetical Stanley formulas.
It seems that there is a correspondence between some natural functions on the
set of (shifted) Young diagrams and natural classes of bicolored maps. Is there
some general pattern? Are there some natural classes of maps which are missing
in this table?

\subsection{The main result 3. Kerov--Olshanski algebra: linear vs spin}

\subsubsection{Kerov--Olshanski algebra} 

The usual (linear) Kerov--Olshanski algebra $\KO$ 
\cite{Kerov1994,Hora2016} (also known under the less compact name \emph{algebra
of polynomial functions on the set of Young diagrams}) is an important tool in
the (linear) asymptotic representation theory of the symmetric groups. One of
its advantages comes from the fact that it can be characterized in several
equivalent ways (for example as the algebra $\Lambda^*$ of \emph{shifted
symmetric functions}); it also has several convenient linear and algebraic bases
which are related to various viewpoints and aspects of the asymptotic
representation theory.

For the purposes of the current paper, Kerov--Olshanski algebra 
\[ \KO := \mathrm{span}\{ \Ch_\class : \class \in \Part \} \]
may be defined as
the linear span of the normalized linear characters of the symmetric groups. 

\subsubsection{Spin Kerov--Olshanski algebra}

We define the \emph{spin} Kerov--Olshanski algebra 
(maybe \emph{Ivanov algebra} would be an even better name)
\begin{equation}
\label{eq:Gamma}
 \Gamma := \mathrm{span}\{ \ChSpin_\class : \class \in \SP \} 
 \end{equation}
as the linear span of spin characters \cite[Section 6]{Ivanov2004}.
Ivanov proved  
that the elements of $\Gamma$ can be identified with 
\emph{supersymmetric polynomials}, thus $\Gamma$ is a unital, commutative algebra. 

\subsubsection{Double of a function. Kerov--Olshanski algebra: linear vs spin} 

If $F\colon \Part\to \Q$ is a function on the
set of partitions, we define its \emph{double} as the function $\double F \colon \SP \to \Q$ on the
set of \emph{strict} partitions
\[ \left( \double F \right)(\xi) := F\big( D(\xi) \big) \qquad \text{for }\xi\in\SP\]
given by doubling of the argument.

\medskip

The following simple result reduces some questions about the spin
Kerov--Olshanski algebra $\Gamma$ to its linear counterpart $\KO$.
\begin{theorem}
	\label{prop:isomorphism}
	The map 
	\begin{equation}
	\label{eq:isomorphism}
	\double \colon \KO/\ker \double \to \Gamma
	\end{equation}
	is an isomorphism of algebras.
\end{theorem} 
The proof is postponed to \cref{sec:proof-of-isomorphism}.

\subsubsection{Filtrations: linear vs spin}
\label{sec:filtrations}

In applications it is often convenient to equip Kerov--Olshanski algebra with
this or another filtration which is tailored for the specific asymptotic regime
one is interested in; several such filtrations were considered in the
literature. In order to stay focused we will concentrate on a specific choice of
such a filtration $\F_0\subseteq \F_1 \subseteq \cdots \subseteq \KO$ given by
\begin{equation} 
\label{eq:filtration}
\F_k := \mathrm{span}\{ \Ch_\class : \class \in \Part,\quad  |\class|+\ell(\class)\leq k \},
\end{equation}
where $\ell(\class)=\ell$ denotes the number of parts of
$\class=(\class_1,\dots,\class_\ell)$. 
This specific choice is motivated by investigation of asymptotics of (random)
Young diagrams and tableaux in the scaling in which they grow to infinity in
such a way that they remain \emph{balanced} \cite{Biane1998,Sniady2006}.

\medskip

We define its spin counterpart as the family of vector spaces
$\G_0\subseteq \G_1 \subseteq \cdots \subseteq \Gamma$ defined by
\begin{equation} 
\label{eq:filtration2}
\G_k := \mathrm{span}\{ \ChSpin_\class : \class \in \OP,\quad  |\class|+\ell(\class)\leq k \}.
\end{equation}
Note that $|\class|+\ell(\class)$ is always an even integer for any $\class \in
\OP$; it follows therefore that $\G_{2k+1}=\G_{2k}$ holds for any integer $k\geq 0$.
Informally speaking, this means that only the \emph{`even'} part of the family
$\G_0\subseteq \G_1 \subseteq \cdots \subseteq \Gamma$ contains interesting
information.

\medskip

The following result provides a direct link between the families $(\F_i)$ and
$(\G_i)$ via the isomorphism \eqref{eq:isomorphism}. Also its proof (postponed
to \cref{sec:proof-of-filtration}) makes use of this isomorphism.
\begin{theorem}
	\label{thm:filtration}
	The family $(\G_k)$ is a filtration on the algebra $\Gamma$.
	
	Furthermore, $\G_k = \double(\F_k)$ for any integer $k\geq 0$.
\end{theorem}

\subsubsection{Application. Gaussian fluctuations of partitions: linear vs spin}

Our motivation for studying the filtration $(\G_i)$ comes from investigation of
random strict partitions related to projective asymptotic representation theory
of the symmetric groups. \cref{thm:filtration} plays a prominent role in our
forthcoming paper \cite{Sniady2019} which is devoted to this topic.

A convenient tool for proving Gaussianity of fluctuations of random partitions
is \emph{approximate factorization property for characters}
\cite{Sniady2006,DolegaSniady2018} which is formulated in the language of
certain \emph{cumulants}. In the case of the \emph{linear} characters of the
symmetric groups this property was known to be true \cite{Sniady2006}.

In the aforementioned forthcoming paper \cite{Sniady2019} we will show how
to reformulate the results of the current paper in a more abstract way which
allows to relate the cumulants for \emph{linear} characters to their \emph{spin}
counterparts. In this way we will show how the approximate factorization
property of \emph{linear} characters of the symmetric groups implies directly
its \emph{spin} counterpart.

\subsection{Structure of the paper}

\cref{sec:spin-characters} summarizes some fine issues related to the definition of the spin characters.

In \cref{sec:spin-in-terms-of-linear} we prove our starting tool,
\cref{thm:general_characters}, which expresses the linear characters $(\Ch_\pi)$
in terms of the spin characters $(\ChSpin_\pi)$. As a consequence, we also prove
\cref{theo:spin-in-linear} which gives roughly the opposite: the spin characters
$(\ChSpin_\pi)$ in terms of the linear characters $(\Ch_\pi)$.

\cref{sec:proof-of-stanley} is devoted to the proof of \cref{thm:spin-Stanley},
i.e.~spin Stanley character formula.
\cref{sec:proof-of-maps} is devoted to the proof of \cref{thm:maps}, i.e.~spin Stanley formula in terms of non-oriented, orientable maps.

In \cref{sec:multirectangular} we recall the shifted version of
\emph{multirectangular coordinates} and in \cref{sec:asymptotic-stanley} we
apply these coordinates for asymptotics of spin characters. In
\cref{sec:proof-of-isomorphism} we apply these results in order to prove
\cref{prop:isomorphism,thm:filtration} which provide the link between (the
filtration on) the linear Kerov--Olshanski algebra and its spin counterpart. 

\medskip

This article is the full version of a $12$-page \emph{extended abstract}
\cite{MatsumotoSniadyFPSAC} which will be published in the proceedings of the conference
\emph{Formal Power Series and Algebraic Combinatorics 2019}. 

\section{Spin characters, spin representations}
\label{sec:spin-characters}

In order to keep the Introduction lightweight we decided to postpone the discussion of some subtle technical issues related to the definition of spin characters 
and spin representations until the current section.
Our presentation is based on 
\cite{Stembridge1989,Wan2012,Kleshchev2005,Ivanov2004}.

\subsection{Conjugacy classes of $\Spin{n}$}

We denote by $\SP_n^+$ (respectively, $\SP_n^-$) the set of 
strict partitions $\xi\in\SP_n$ with the property that
$n-\ell(\xi)$ is even (respectively, odd).

\medskip

For a partition $\pi\vdash n$ we denote by $C_\pi\subset \Spin{n}$ the set of elements of the spin group which are mapped
--- under the canonical homomorphism $\Spin{n}\to\Sym{n}$ ---
to permutations with the cycle-type given by $\pi$.

\medskip

Schur \cite{Schur1911} 
proved the following dichotomy for $\pi\vdash n$:
\begin{itemize}
	\item if one of the following two conditions is fulfilled:
	\begin{itemize}[label=\ding{212}]
		\item $\pi\in \OP_n$, or
		\item $\pi\in \SP_n^-$
	\end{itemize}
	then $C_\pi$ splits into a pair of conjugacy classes of $\Spin{n}$ which will be denoted by $C_\pi^\pm$;
	
	\item otherwise, $C_\pi$ is a conjugacy class of $\Spin{n}$. 
	
\end{itemize}

\subsection{Conjugacy classes and spin characters}

Any spin character 
vanishes on the conjugacy class $C_\pi$ which does not split, cf.~\cite[p.~95]{Stembridge1989}.
For this reason, \emph{from the viewpoint of the spin character theory only the conjugacy classes $C_\pi^\pm$ are interesting}.

\medskip

Spin representations are exactly the ones which map the central element $z\in\Spin{n}$ to $-\operatorname{Id}\in\GL(V)$.
Since $C_\pi^- = z C_\pi^+$, it follows that
the value of any spin character on $C_\pi^-$ is the opposite of its value on $C_\pi^+$. For this reason, 
\emph{from the viewpoint of the spin character theory the conjugacy classes $C_\pi^-$ are redundant and it is enough to consider the character values only on the conjugacy classes $C_\pi^+$}.

\medskip

From the viewpoint of the asymptotic representation theory it is natural to consider some \emph{sequence} of groups together  with some natural inclusions; in our case this is the sequence 
\[ \Spin{1} \subset \Spin{2} \subset \Spin{3} \subset \cdots \]
of spin groups. Such a setup allows to relate a conjugacy class of a smaller group to some conjugacy class in the bigger group and, in this way, to evaluate the irreducible characters of the bigger group on the conjugacy classes of the smaller one.

Regretfully, the conjugacy classes $C_\pi^+$ which correspond to $\pi\in\SP_n^-$ do not behave nicely under such inclusions. 
Indeed, on the level of the symmetric groups the inclusion $\Sym{n}\subset \Sym{n+k}$ corresponds to adding $k$ fixpoints to a given permutation; in other words the set $C_\pi\subset \Spin{n}$ corresponds to $C_{\pi,1^k}\subset \Spin{n+k}$ and the latter does not split because $(\pi,1^k)\notin \SP_n$ (at least for $k\geq 2$) and $(\pi,1^k)\notin \OP_n$ (because $n-\ell(\pi)$ is odd which implies that at least one part of $\pi$ is even).

For this reason, \emph{for the purposes of the asymptotic representation theory it is enough to consider only the conjugacy classes $C_\pi^+$ for $\pi\in\OP_n$.}

\subsection{Irreducible spin representations}
\label{sec:irreducible-spin-representations}

The relationship between strict partitions and the irreducible spin
representations of the symmetric groups is \emph{not} a bijective one.
Nevertheless, as we shall discuss below, this non-bijectivity can be ignored to
large extent.

\medskip

More specifically (see \cite[p.~235]{Schur1911} and \cite[Theorem
7.1]{Stembridge1989}), each $\xi\in\SP_n^+$ corresponds to a \emph{single}
irreducible representation.
We denote by its character by $\phi^{\xi}$.

\medskip

On the other hand, each $\xi\in\SP_n^-$ corresponds to a \emph{pair} of
irreducible spin representations; we denote their characters by $\phi^{\xi}_+$
and $\phi^{\xi}_-$. These two characters coincide on the conjugacy classes
$C^{\pm}_{\pi}$ over $\pi\in\OP_n$. For the purposes of the current paper we not
need to evaluate the characters on $C^{\pm}_{\pi}$ for $\pi\in \SP_n^-$; for
this reason we do not have to distinguish between them and we may denote them by
the same symbol $\phi^\xi$.

\subsection{Spin characters: conclusion}
\label{appendix:conclusion}

For all partitions $\xi\in\SP_n$, $\pi\in\OP_n$ the
value of the projective character
\[ \phi^\xi(\pi)= \Tr \irrepSp^\xi(c^\pi) \]
is well defined, where $c^\pi\in C_\pi^+$ is a representative of the of the conjugacy class $C_\pi^+$, cf.~\cite[Eq.~(2.1)]{Stembridge1989}.

\subsection{Characters $\XX{\xi}{\pi}$}

Following Ivanov \cite[Section 2]{Ivanov2004}, given $\xi\in\SP_n$ we define a function on $\OP_n$
\[ \tilde{\phi}^\xi=
\begin{cases}
\phi^\xi & \text{if } \xi \in \SP_n^+, \\
\frac{\phi^\xi_+ + \phi^\xi_-}{\sqrt{2}}= 
\sqrt{2}\ \phi^\xi & \text{if } \xi \in \SP_n^-. 
\end{cases}
\]
In the following it will be more convenient to pass to quantities
\[\XX{\xi}{\pi}:=2^{\frac{\ell(\xi)-\ell(\pi)}{2}}\ \PHI{\xi}{\pi},\]
cf.~\cite[Proposition 3.3]{Ivanov2004}. 
Alternatively, we can define the $X^\xi(\pi)$ by the formula
\[
p_\pi = \sum_{\xi \in \SP_n} X^\xi(\pi) P_\xi,
\]
where $p_\pi$ is Newton's powers-sums and $P_\xi$ is the Schur's $P$-function,
as in \cite[Example III.8.11(c)]{Macdonald1995}.

With these notations
\eqref{eq:projective-normalized} can be rewritten as
\[
\ChSpin_\class(\xi):=\begin{cases}
n^{\downarrow k}\ 
\ \frac{ \XX{\xi}{\class\cup 1^{n-k}}}{\XX{\xi}{1^{n}}} = 
n^{\downarrow k}\ 2^{\frac{k-\ell(\class)}{2}}
\ \frac{ \PHI{\xi}{\class\cup 1^{n-k}}}{ \PHI{\xi}{1^{n}}}
& \text{if } n\geq k, \\
0 & \text{otherwise.}
\end{cases}
\]

\section{Spin characters in terms of linear characters}
\label{sec:spin-in-terms-of-linear}

\subsection{The main result 1} 
We were not able to find this result in the
literature and we believe it is new. Note, however, that this result is
essentially only a reformulation of the classical equality
\eqref{eq:Schur_SchurQ} which gives a relation between Schur polynomial
$s_{D(\xi)}$ and Schur's $Q$-function $Q_\xi$.

\begin{theorem}[The main result 1: linear in terms of spin] 
\label{thm:general_characters}
	For any $\class\in\OP$ and $\xi\in\SP$,
	\begin{equation}
	\label{eq:spin-vs-linear}
	\Ch_\class\big( D(\xi)\big) = 
	\sum_{\sett \subseteq [\ell(\class)]} 
	\ChSpin_{\class(\sett)}(\xi)\  \ChSpin_{\class(\sett^c)}(\xi)
	\end{equation}
	where $\class(\sett)=(\class_{i_1}, \class_{i_2},\dots, \class_{i_r})$ for
$\sett=\{i_1<i_2< \cdots< i_r\}$ and $\sett^c=[\ell(\class)]\setminus \sett$
denotes the complement of $\sett$.
\end{theorem}

The proof is presented below in \cref{sec:proof}.

\medskip

Each  summand on the right-hand side of \eqref{eq:spin-vs-linear} is equal to
another summand in which the roles of the set $\sett$ and its complement
$\sett^c$ are interchanged. By grouping such pairs of summands (each
such a pair corresponds to a set-partition of $[\ell(\pi)]$ to \emph{at most}
two blocks) and by dividing both sides of \eqref{eq:spin-vs-linear} by $2$, the
above result can be reformulated as the following equality between functions on
$\SP$:
\begin{equation}
\label{eq:spin-vs-linear2} \ChD_\class = \sum_{\substack{ \partition: \\
		|\partition|\leq 2} } \prod_{b\in \partition} \ChSpin_{(\class_i : i \in b)},
\end{equation}
where the sum runs over set-partitions $\partition$ of the set $[\ell(\class)]$ into
at most two blocks, and where we denote
\[ \ChD_\class(\xi):=\frac{1}{2} \Ch_\class\big( D(\xi) \big)\]
for $\xi\in\SP$ and $\class\in\OP$.

\begin{example}
	\cref{thm:general_characters} implies the following equalities between
functions on the set of strict partitions:
\begin{equation}
\label{eq:example-chd}
\left\{
\begin{aligned}
\ChD_{k_1} &=  
\ChSpin_{k_1}, \\
\ChD_{k_1,k_2} &=  \ChSpin_{k_1,k_2} +
\ChSpin_{k_1}\ 
\ChSpin_{k_2}, \\
\ChD_{k_1,k_2,k_3} &=  \ChSpin_{k_1,k_2,k_3} + \\ &+
\ChSpin_{k_1,k_2}\
\ChSpin_{k_3} + 
\ChSpin_{k_1,k_3}\
\ChSpin_{k_2} 
+
\ChSpin_{k_2,k_3}\ 
\ChSpin_{k_1},   \\
\vdots
\end{aligned} \right.
\end{equation}
for arbitrary odd integers $k_1,k_2,\ldots\geq 1$.
\end{example}

\subsection{Proof of \cref{thm:general_characters}}
\label{sec:proof}

\begin{proof}[Proof of \cref{thm:general_characters}]
For an integer partition $\class$ we denote
\[ z_\class = \prod_{j\geq 1} j^{m_j(\class)} m_j(\class)!, \]
where $m_j(\class)$ denotes the \emph{multiplicity} of $j$ in the partition
$\class$.  

We denote by $f^\lambda=\CHIL{\lambda}{1^{|\lambda|}}$ the number of
standard tableaux of shape $\lambda$.
For a strict partition $\xi$ we denote 
\[g^\xi= \XX{\xi}{1^{|\xi|}} \] 
which also happens to be the number
of \emph{shifted} standard tableaux with the shape given by the shifted Young
diagram $\xi$, see \cite[III-8, Ex.~12]{Macdonald1995}.

\medskip
	
Recall the symmetric function algebra $\Lambda=\Q[p_1,p_2,p_3,\dots]$ and its
subalgebra, \emph{the algebra of supersymmetric functions}
$\Gamma=\Q[p_1,p_3,p_5,\dots]$, where the $p_r$ are Newton's power-sums.
Define the algebra homomorphism $\varphi:\Lambda \to \Gamma$ by
\[
\varphi(p_r)= \begin{cases}
2 p_r & \text{if $r$ is odd}, \\
0 & \text{if $r$ is even}.
\end{cases}
\]
Then \cite[III-8, Ex.~10]{Macdonald1995} implies that for any strict partition
$\xi$ we have
\begin{equation} 
\label{eq:Schur_SchurQ}
\varphi(s_{D(\xi)})= 2^{-\ell(\xi)}\ (Q_\xi)^2,
\end{equation}
where $Q_\zeta=Q_\zeta(x; -1)$ denotes Schur's $Q$-function
\cite[III-8]{Macdonald1995}.

Recall the Frobenius formula for Schur functions:
\[
s_{\mu}= \sum_{\class} z_\class^{-1}\ \CHIL{\mu}{\class}\ p_{\class}.
\]
Applying the homomorphism $\varphi$ to this identity with
$\mu=D(\xi)$, we obtain
\[
\varphi(s_{D(\xi)})= \sum_{\class \in \OP_{2n}} 2^{\ell(\class)}
z_\class^{-1}\ \CHIL{D(\xi)}{\class}\ p_\class.
\]
And, recall the Frobenius formula for Schur $Q$-functions:
\[
Q_\xi = \sum_{\nu \in \OP_n} 
2^{\ell(\nu)} z_\nu^{-1}\ \XX{\xi}{\nu}\ p_\nu.
\]
Substituting the above two formulas to \eqref{eq:Schur_SchurQ},
we have for any $\xi\in\SP_n$
\begin{equation} \label{eq:Schur_SchurQ2}
\sum_{\class \in \OP_{2n}} 
2^{\ell(\class)}
z_\class^{-1}\ \CHIL{D(\xi)}{\class}\ p_\class
= 2^{-\ell(\xi)} 
\left( \sum_{\nu \in \OP_n} 
2^{\ell(\nu)} z_\nu^{-1}\ \XX{\xi}{\nu}\ p_\nu\right)^2.
\end{equation}
By comparing the coefficients of $p_{(1^{2n})}=p_{(1^{n})}p_{(1^{n})}$ in both
sides of \eqref{eq:Schur_SchurQ2}, we find 
\begin{equation} \label{eq:f_to_g2}
\frac{f^{D(\xi)}}{(2n)!}  = 
2^{-\ell(\xi)} \left(\frac{g^\xi}{n!}\right)^2.
\end{equation}
For an alternative proof of this identity see \cite[Proposition 3.1]{Linusson2018}.

\bigskip

\emph{First we assume that $\class$ is an odd partition
which does not have parts equal to $1$}, i.e., 
$m_1(\class)=0$.
By comparing the coefficients of $p_{\class \cup (1^{2n-|\class|})}$ in both
sides of \eqref{eq:Schur_SchurQ2} we find
\begin{multline*}
\frac{\CHIL{D(\xi)}{\class\cup (1^{n-|\class|})}}{z_{\class\cup
		(1^{n-|\class|})}} = \\  
	2^{-\ell(\xi)} \sum_{\substack{ \mu^1, \mu^2 \\ \mu^1
		\cup \mu^2 =\class}} \frac{ \XX{\xi}{\mu^1 \cup (1^{n-|\mu^1|})}}{z_{\mu^1
		\cup (1^{n-|\mu^1|})}} \frac{ 
	\XX{\xi}{\mu^2 \cup (1^{n-|\mu^2|})}}{z_{\mu^1
		\cup (1^{n-|\mu^2|})}}.
\end{multline*}
By the assumption $m_1(\class)=0$, we have
$z_{\class \cup (1^{2n-|\class|})}= z_\class \cdot (2n-|\class|)!$
and  $z_{\mu^i \cup (1^{n-|\mu^i|})}= z_{\mu^i} \cdot (n-|\mu^i|)!$.
Thus, we obtain
\begin{multline*}
\frac{\CHIL{D(\xi)}{\class\cup (1^{n-|\class|})}}{z_\class \cdot (2n-|\class|)!}
= \\ 2^{-\ell(\xi)} \sum_{\substack{ \mu^1, \mu^2  \\
		\mu^1 \cup \mu^2 =\class}}
\frac{ \XX{\xi}{\mu^1 \cup (1^{n-|\mu^1|})}}{z_{\mu^1} \cdot (n-|\mu^1|)!} 
\frac{ \XX{\xi}{\mu^2 \cup (1^{n-|\mu^2|})}}{z_{\mu^2} \cdot (n-|\mu^2|)!}.
\end{multline*}
Taking the quotient of this and \eqref{eq:f_to_g2}, 
we have
\begin{multline*}
\frac{1}{z_\class} \frac{(2n)!}{(2n-|\class|)!} 
\frac{\CHIL{D(\xi)}{\class \cup(1^{2n-|\class|})}}{f^{D(\xi)}}
=\\ \sum_{\substack{ \mu^1, \mu^2  \\
		\mu^1 \cup \mu^2 =\class}}
\frac{1}{z_{\mu^1} z_{\mu^2}} 
\frac{n!}{(n-|\mu^1|)!} \frac{\XX{\xi}{\mu^1 \cup (1^{n-|\mu^1|})}}{g^\xi} \times \\ \times
\frac{n!}{(n-|\mu^2|)!} \frac{\XX{\xi}{\mu^2 \cup (1^{n-|\mu^2|})}}{g^\xi},
\end{multline*}
which is equivalent to 
\[
\Ch_\class\big(D(\xi)\big)= 
\sum_{\substack{ \mu^1, \mu^2  \\
		\mu^1 \cup \mu^2 =\class}}
\frac{z_\class}{z_{\mu^1} z_{\mu^2}} 
\ChSpin_{\mu^1}(\xi) \ChSpin_{\mu^2}(\xi).
\]
It is easy to see that
this is equivalent to the desired formula.
Thus, we completed the proof of the theorem under the assumption 
$m_1(\class)=0$.

	\bigskip
	
	\emph{Next we consider a general odd partition $\class$} and we
	write it as
	$\class=\tilde{\class} \cup (1^r)$,
	where $m_1(\tilde{\class})=0$ and $r=m_1(\class)$.
	Then 
	\begin{multline}
	\label{eq:to-be-continued}
	\sum_{\sett \subseteq [\ell(\class)] } \ChSpin_{\class(\sett)} (\xi)
	\ChSpin_{\class(\sett^c)}(\xi) \\
	= \sum_{\setJ \subseteq [\ell(\tilde{\class})]}
	\sum_{s=0}^r \binom{r}{s}
	\ChSpin_{\tilde{\class}(\setJ) \cup (1^s)}(\xi)
	\ChSpin_{\tilde{\class}(\setJ^c) \cup (1^{r-s})}(\xi),
	\end{multline}
	where $\setJ^c= [\ell(\tilde{\class})] \setminus \setJ$.
	By virtue of the identity 
	\begin{equation} \label{eq:reductions}
	\ChSpin_{\nu \cup (1^s)}(\xi) = 
	(n-|\nu|)^{\downarrow s}\ \ChSpin_\nu (\xi),
	\end{equation}
	we have 
	\begin{multline*}
	\eqref{eq:to-be-continued} = \sum_{\setJ \subseteq [\ell(\tilde{\class})] }
	\ChSpin_{\tilde{\class}(\setJ)}(\xi)
	\ChSpin_{\tilde{\class}(\setJ^c)}(\xi) \times \\ \times
	\sum_{s=0}^r \frac{r!}{s! (r-s)!} 
	\frac{(n-|\tilde{\class}(\setJ)|)!}{(n-|\tilde{\class}(\setJ)|-s)!}
	\frac{(n-|\tilde{\class}(\setJ^c)|)!}{(n-|\tilde{\class}(\setJ^c)|-(r-s))!}.
	\end{multline*}
	Here it is easy to see that
	\begin{multline*}
	 \sum_{s=0}^r \frac{r!}{s! (r-s)!} 
	\frac{(n-|\tilde{\class}(\setJ)|)!}{(n-|\tilde{\class}(\setJ)|-s)!}
	\frac{(n-|\tilde{\class}(\setJ^c)|)!}{(n-|\tilde{\class}(\setJ^c)|-(r-s))!} \\
	= r! \sum_{s=0}^r \binom{n-|\tilde{\class}(\setJ)|}{s}
	\binom{n-|\tilde{\class}(\setJ^c)|}{r-s} \\
	= r! \binom{2n-|\tilde{\class}(\setJ)| -|\tilde{\class}(\setJ^c)|}{r} 
	= \frac{(2n-|\tilde{\class}|)!}{(2n-|\tilde{\class}|-r)!}. 
	\end{multline*}
	Thus we have obtained 
	\begin{align*}
	\sum_{\sett \subseteq [\ell(\class)]} \ChSpin_{\class(\sett)} (\xi)
	\ChSpin_{\class(\sett^c)}(\xi)
	=& \frac{(2n-|\tilde{\class}|)!}{(2n-|\tilde{\class}|-r)!} 
	\sum_{\sett \subseteq [\ell(\tilde{\class})]}
	\ChSpin_{\tilde{\class}(\setJ)}(\xi)
	\ChSpin_{\tilde{\class}(\setJ^c)}(\xi) \\
	=& \frac{(2n-|\tilde{\class}|)!}{(2n-|\tilde{\class}|-r)!} 
	\Ch_{\tilde{\class}}\big(D(\xi)\big) \\
	=& \Ch_\class \big(D(\xi)\big)
	\end{align*}
	which concludes the proof.
	Here the second equality follows from the previous part of the present proof
	and the third equality follows from \eqref{eq:reductions}.
\end{proof}

\subsection{The main result 2. Spin characters in terms of linear characters}

Formulas \eqref{eq:example-chd} can be viewed as an upper-triangular system of
equations with unknowns $(\ChSpin_{\class})_{\class\in\OP}$. It can be solved, for
example
\begin{equation}
\label{eq:example-chd2}
\left\{
\begin{aligned}
\ChSpin_{k_1} &=  
\ChD_{k_1}, \\[1ex]
\ChSpin_{k_1,k_2} &=  \ChD_{k_1,k_2} - \ChD_{k_1}\ 
\ChD_{k_2}, \\[1ex]
\ChSpin_{k_1,k_2,k_3} &=  \ChD_{k_1,k_2,k_3} \\ & -
\ChD_{k_1,k_2}\
\ChD_{k_3} 
-
\ChD_{k_1,k_3}\
\ChD_{k_2} 
-
\ChD_{k_2,k_3}\ 
\ChD_{k_1} 
\\ &
+ 3 \ChD_{k_1} \ChD_{k_2} \ChD_{k_3},  \\
\vdots
\end{aligned} \right.
\end{equation}

The general pattern is given by the following result.

\begin{theorem}[The main result 2: spin in terms of linear]
	\label{theo:spin-in-linear} For any $\class\in\OP$ the following equality
between functions on the set $\SP$\! of strict partitions holds true:
\begin{equation} 
\label{eq:spin-in-linear}
\ChSpin_\class = \sum_{\partition} (-1)^{|\partition|-1}\ (2|\partition|-3)!!\
\prod_{b\in \partition} \ChD_{(\class_i:i \in b)}, 
\end{equation}
where the sum runs over all set-partitions of the set 
$[\ell(\class)]$ and where we use the convention that $(-1)!!=1$.
\end{theorem}
\begin{proof}
By singling out the partition $\partition$ in \eqref{eq:spin-vs-linear2} which
consists of exactly one block we may express the spin character $\ChSpin_\class$ in
terms of the linear character $\ChD_\class$ and spin characters $\ChSpin_{\class'}$
which correspond to partitions $\class'\in\OP$ with $\ell(\class')<\ell(\class)$:
\begin{equation}
\label{eq:spin-vs-linear3}
\ChSpin_\class= \ChD_{\class} -
\sum_{\substack{ \partition: \\ |\partition|= 2} } \prod_{b\in \partition} \ChSpin_{(\class_i : i \in b)}.
\end{equation}
By applying this procedure recursively to the spin characters on the right-hand
side, we end up with an expression for $\ChSpin_\class$ as a linear combination (with
integer coefficients) of the products of the form
\begin{equation} 
\label{eq:mysummand}
\prod_{b\in \partition} \ChD_{(\class_i : i \in b)} 
\end{equation}
over set-partitions $\partition$ of $[\ell(\class)]$. The remaining difficulty is to
determine the exact value of the coefficient of \eqref{eq:mysummand} in this
linear combination.

\medskip

The above recursive procedure can be encoded by a tree as follows. Each vertex
which is a leaf is labelled by some linear character $\ChD_{(\class_i : i \in
	c)}$ for some non-empty subset $c\subseteq [\ell(\class)]$ or --- to keep the
notation light --- by the set $c$. Each vertex which is not a leaf is labelled
by some spin character $\ChSpin_{(\class_i : i \in c )}$ for some non-empty
subset $c\subseteq [\ell(\class)]$ or --- to keep the notation light --- by the
set $c$.

Each vertex $c\subseteq [\ell(\class)]$ which is not a leaf
has exactly two children $c_1,c_2\subseteq [\ell(\class)]$ which are non-empty,
disjoint sets such that $c=c_1 \sqcup c_2$. 
Thus the labels of all non-leafs are
uniquely determined by the labels of the leafs, so we may remove these non-leaf
labels. Since we are interested in the coefficient of \eqref{eq:mysummand}, we
require that the set of leaf labels is equal to the set of blocks of $\partition$
from \eqref{eq:mysummand}.

The resulting non-ordered trees with the property that each non-leaf vertex has
exactly two children are known as \emph{total binary partitions}; the
cardinality of such trees with the prescribed set $\partition$ of leaf labels is
equal to $(2 |\partition| - 3)!!$, cf.~\cite[Example 5.2.6]{Stanley1999}.

\medskip

Our recursive procedure involves change of the sign; such a change occurs once
for each non-leaf vertex. Thus each total binary tree contributes with
multiplicity $(-1)^{|\partition|-1}$ which concludes the proof.
\end{proof}

\section{Proof of spin Stanley character formula}

\label{sec:proof-of-stanley}

\begin{proof}[Proof of \cref{thm:spin-Stanley}]
We start with \cref{theo:spin-in-linear} and substitute each normalized linear
character 
\[ \ChD_{(\class_i:i \in b)}(\xi)=\frac{1}{2} \Ch_{(\class_i:i \in b)}\big( D(\xi)  \big)\]
which contributes to the right-hand side of
\eqref{eq:spin-in-linear} by the linear Stanley character formula
\eqref{eq:stanley-linear}. 

We shall discuss in detail the case when $\pi=(\pi_1,\pi_2)$ consists of just
two parts. We will view $\Sym{\pi_1}$, $\Sym{\pi_2}$ and $\Sym{\pi_1+\pi_2}$ as
the groups of permutations of, respectively, the set $\{1,\dots,\pi_1\}$,
$\{\pi_1+1,\dots,\pi_1+\pi_2\}$ and $\{1,\dots,\pi_1+\pi_2\}$. In this way we
may identify $\Sym{\pi_1}\times\Sym{\pi_2}$ as a subgroup of
$\Sym{\pi_1+\pi_2}$. As usually, we identify $(\pi_1)\in\Sym{\pi_1}$,
$(\pi_2)\in\Sym{\pi_2}$  with arbitrary permutations with prescribed cycle
structures; then $(\pi_1,\pi_2):=\big( (\pi_1), (\pi_2) \big) \in \Sym{\pi_1}
\times \Sym{\pi_2} \subseteq \Sym{\pi_1+\pi_2}$ is also a permutation with
appropriate cycle structure. With these notations we have
\begin{multline} 
\label{eq:two-parts}
\ChSpin_{(\pi_1,\pi_2)}(\xi) =  \\
\shoveright{ \frac{(-1)!!}{2} \Ch_{(\pi_1,\pi_2)} \big( D(\xi) \big) -
\frac{1!! }{2^2} \Ch_{(\pi_1)} \big( D(\xi) \big) 
\ \Ch_{(\pi_2)} \big( D(\xi) \big) =} \\
\shoveleft{\frac{(-1)!!}{2} 
\sum_{\substack{\sigma_1,\sigma_2\in\Sym{\pi_1+\pi_2} \\
	\sigma_1\sigma_2=(\pi_1,\pi_2)}}
(-1)^{\sigma_1} N_{\sigma_1,\sigma_2}\big( D(\xi) \big)
} 
\\ 
-\frac{1!! }{2^2}
\sum_{\substack{\sigma_1^{(1)},\sigma_2^{(1)} \in\Sym{\pi_1} \\
		\sigma^{(1)}_1\sigma^{(1)}_2=(\pi_1)}}
(-1)^{\sigma_1^{(1)}} N_{\sigma_1^{(1)},\sigma_2^{(1)}}
\big( D(\xi) \big)  \times 
\\
\shoveright{\times 
\sum_{\substack{\sigma_1^{(2)},\sigma_2^{(2)} \in\Sym{\pi_2} \\
		\sigma^{(2)}_1\sigma^{(2)}_2=(\pi_2)}}
(-1)^{\sigma_1^{(2)}} N_{\sigma_1^{(2)},\sigma_2^{(2)}}
\big( D(\xi) \big)=}
\\ 
\shoveleft{\frac{(-1)!!}{2} 
	\sum_{\substack{\sigma_1,\sigma_2\in\Sym{\pi_1+\pi_2} \\
			\sigma_1\sigma_2=(\pi_1,\pi_2)}}
	(-1)^{\sigma_1} N_{\sigma_1,\sigma_2}\big( D(\xi) \big)
	} 
\\ 
-\frac{1!! }{2^2}
\sum_{\substack{\sigma_1,\sigma_2\in\Sym{\pi_1}\times\Sym{\pi_2} 	\\ \sigma_1\sigma_2=(\pi_1,\pi_2)}}
(-1)^{\sigma_1} N_{\sigma_1,\sigma_2}
\big( D(\xi) \big),
\end{multline}
where the last equality follows from identification between the pair
$(\sigma_i^{(1)}, \sigma_i^{(2)})$ with the corresponding
permutation $\sigma_i\in\Sym{\pi_1}\times \Sym{\pi_2}$. 
\medskip

In general, 
\begin{equation} 
\label{eq:Stanley-not-ready}
\ChSpin_{\pi}(\xi) = 
\sum_{\substack{\sigma_1,\sigma_2\in\Sym{|\pi|} \\
		\sigma_1\sigma_2=\pi}} c_{\sigma_1,\sigma_2}\ (-1)^{\sigma_1}\ N_{\sigma_1,\sigma_2}\big(D(\xi)\big)
	\end{equation}
for some combinatorial factor $c_{\sigma_1,\sigma_2}$. For example, in the special
case $\pi=(\pi_1,\pi_2)$ which consists of two parts, \eqref{eq:two-parts}
implies that
\[ c_{\sigma_1,\sigma_2}=\begin{cases}
\frac{(-1)!!}{2} - \frac{1!!}{2^2} & \text{if } \sigma_1,\sigma_2\in \Sym{\pi_1}\times \Sym{\pi_2}, \\[1ex]
\frac{(-1)!!}{2} & \text{otherwise}.
\end{cases}
\]

\medskip

We claim that in the general case the value of the constant $c_{\sigma_1,\sigma_2}$ is equal to  
 \begin{equation}	
 \label{eq:stirling}
 \const_m:= c_{\sigma_1,\sigma_2} = (-1) \sum_p \stirlingS{m}{p} \left(-\frac{1}{2}\right)^p (2p-3)!!, 
 \end{equation}
where $m=|\sigma_1 \vee \sigma_2|$ is the number of orbits 
in the set $[|\pi|]$ under the action of the group $\langle
\sigma_1,\sigma_2\rangle$ generated by $\sigma_1,\sigma_2$,
and $\stirlingS{m}{k}$ denotes Stirling numbers of the second kind. Indeed, the
set-partition $\partition$ (over which we sum in \eqref{eq:spin-in-linear}) can
be identified with a set-partition on the set of the cycles of the permutation
$\pi\in\Sym{|\pi|}$. With this in mind we see that to $c_{\sigma_1,\sigma_2}$
contribute only these set-partitions $\partition$ on the right-hand side of
\eqref{eq:spin-in-linear} for which $\partition$ is bigger than the
set-partition given by the orbits of $\langle \sigma_1,\sigma_2\rangle$. The
collection of such set-partitions can be identified with the collection of
set-partitions of an $m$-element set (i.e.~the set of orbits of $\langle
\sigma_1,\sigma_2\rangle$). For a fixed value $p:=|I|$ of the number of the
blocks there are clearly $\stirlingS{m}{p}$ choices of $I$; each contributes to
$c_{\sigma_1,\sigma_2}$ with the multiplicity
\[ (-1)^{p-1} \ (2p-3)!!\ \frac{1}{2^p}; \]
this concludes the proof of \eqref{eq:stirling}. 

The sum in \eqref{eq:stirling} can be probably calculated by a clever trick
which we, regretfully, failed to find. From our perspective the exact form of
the right-hand side of \eqref{eq:stirling} is not important; in the following we
will make use only of the observation that $c_{\sigma_1,\sigma_2}=\const_m$ depends
only on the number of the orbits of $\langle \sigma_1,\sigma_2\rangle$.

\medskip

In order to evaluate $\const_m$ we shall consider \eqref{eq:Stanley-not-ready} in the
special case of $\pi=1^m$. In this case $\sigma_2=\sigma_1^{-1}$; we denote by
$l=|C(\sigma_1)|$ the number of cycles of $\sigma_1$. It follows that
\[ \ChSpin_{1^m}(\xi)=n^{\downarrow m} = 
\sum_l \stirlingF{m}{l}\ \const_l\ (-1)^{m-l}\ (2n)^{l},
\]
where $n=|\xi|$ and $\stirlingF{m}{l}$ denotes Stirling number of the first
kind. Both sides of the equality are polynomials in the variable $n$; by
comparing the leading coefficients (which corresponds to setting $l=m$) we
conclude that
\[ \const_m = \frac{1}{2^m}.\]
By substituting this value to \eqref{eq:Stanley-not-ready} we conclude the proof.
\end{proof}

\section{Proof of \cref{thm:maps}}
\label{sec:proof-of-maps}

Maps which we consider come in two distinct flavours: \emph{non-oriented maps} versus
\emph{oriented maps}. We review their difference in the following.

\subsection{Non-oriented maps, revisited} 
\label{sec:non-oriented-revisited}

We keep notations from \cref{sec:maps,sec:sum-nonoriented}. The polygons from
the collection $\mathcal{P}$ have labelled edges. These labels can be erased and
we would be still able to recover them if we preserve the following information:
\begin{itemize}
	\item we paint the polygons from the collection $\mathcal{P}$ with distinct colors; 
	
	\item on each polygon from $\mathcal{P}$ we select some white vertex
	(\emph{``root''}); 
	
	\item on each polygon from $\mathcal{P}$ we select one of the two edges
	adjacent to the root (\emph{``direction of rotation\footnote{Even though the
			name \emph{orientation} would sound more appropriate here, we decided to
			reserve the latter word for the context of oriented maps.}''}).
\end{itemize}
The roots and the directions of rotation were indicated on \cref{fig:polygons}
by arrows and harpoons respectively.

\medskip

We recall that the maps which we consider are obtained by gluing the polygons
with \emph{labelled} edges. As a result,  each edge of the map carries two
labels; one on each side of the edge, see \cref{fig:map-projective-plane}. We
can erase these labels and still be able to recover them if we preserve the
following information:
\begin{itemize}
	\item we paint each 
	face of the map with the colour of the corresponding polygon from 
	$\mathcal{P}$ (so that we know the correspondence between the faces of the map and the polygons);
	
	\item on each face of the map we decorate the white corner (\emph{``root''})
	which corresponds to the root of the corresponding polygon from~$\mathcal{P}$;
	
	\item on each face we decorate one of the two edges adjacent to the root which
	corresponds to the direction of rotation of the corresponding polygon from
	$\mathcal{P}$.
\end{itemize}
The roots and the directions of rotation are indicated on
\cref{fig:map-projective-plane} by arrows and harpoons respectively.

\begin{figure}[t]
	\begin{tikzpicture}[scale=0.4,
	black/.style={circle,thick,draw=black,fill=white,inner sep=4pt},
	white/.style={circle,draw=black,fill=black,inner sep=4pt},
	connection/.style={draw=black,black,auto}
	]
	\small
	
	\fill[fill=blue!3]  
	(0*360/14:7) -- (1*360/14:7) -- (2*360/14:7) -- (3*360/14:7) -- (4*360/14:7) -- (5*360/14:7) -- (6*360/14:7) -- (7*360/14:7) -- (8*360/14:7) -- (9*360/14:7) -- (10*360/14:7) -- (11*360/14:7) -- (12*360/14:7) -- (13*360/14:7);
	
	\draw (0*360/14:7)  node (b1)     [black] {};
	\draw[blue, line width=1mm,->] (b1)+(-2,0) to (b1);

	\draw (1*360/14:7)  node (b2)     [white] {};
	\draw (2*360/14:7)  node (b3)     [black] {};
	\draw (3*360/14:7)  node (b4)     [white] {};
	\draw (4*360/14:7)  node (b5)     [black] {};
	\draw (5*360/14:7)  node (b6)     [white] {};
	\draw (6*360/14:7)  node (b7)     [black] {};
	\draw (7*360/14:7)  node (b8)     [white] {};
	\draw (8*360/14:7)  node (b9)     [black] {};
	\draw (9*360/14:7)  node (b10)    [white] {};
	\draw (10*360/14:7) node (b11)     [black] {};
	\draw (11*360/14:7) node (b12)     [white] {};
	\draw (12*360/14:7) node (b13)     [black] {};
	\draw (13*360/14:7) node (b14)     [white] {};

	\draw (-0.5*360/14:6)  node {$1$};
	\draw (-2.5*360/14:6)  node {$2$};
	\draw (-4.5*360/14:6)  node {$3$};
	\draw (-6.5*360/14:6)  node {$4$};
	\draw (-8.5*360/14:6)  node {$5$};
	\draw (-10.5*360/14:6)  node {$6$};
	\draw (-12.5*360/14:6)  node {$7$};

	\draw[connection] (b1) to  (b2);
	\draw[connection] (b2) to  (b3);
	\draw[connection] (b3) to  (b4);
	\draw[connection] (b4) to  (b5);
	\draw[connection] (b5) to  (b6);
	\draw[connection] (b6) to  (b7);
	\draw[connection] (b7) to  (b8);
	\draw[connection] (b8) to  (b9);
	\draw[connection] (b9) to  (b10);
	\draw[connection] (b10)to  (b11);
	\draw[connection] (b11) to (b12);
	\draw[connection] (b12) to (b13);
	\draw[connection] (b13) to (b14);
	\draw[connection] (b14) to (b1);
	
	\draw [ultra thick,<-] (-3,-2) arc(-160:160:1);
	\end{tikzpicture}
	
	\caption{A collection of oriented polygons of face type $(7)$. The labeling of
		the edges corresponds to the permutation $\pi=(1,2,3,4,5,6,7)\in\Sym{7}$. The
		root is indicated by the thick blue arrow. The orientation is indicated by the
		circular arrow.} \label{fig:polygon-oriented}
	
	\vspace{3ex}
	\centering
	\begin{tikzpicture}[scale=0.6,
	white/.style={circle,draw=black,fill=white,inner sep=4pt},
	black/.style={circle,draw=black,fill=black,inner sep=4pt},
	connection/.style={draw=black!80,black!80,auto}
	]
	\footnotesize
	
	\begin{scope}
	\clip (0,0) rectangle (10,10);
	\fill[blue!3] (0,0) rectangle (10,10);
	
	\draw [ultra thick,<-] (4,8) arc(-160:160:1);
	
	\draw (3,5) node (b1) [black] {};
	\draw (b1) +(10,0) node (b1prim) [black] {};
	
	\draw (b1) +(1,-3) node (b1-se) [white] {};
	\draw (b1) +(-1,-3) node (b1-sw) [white] {};
	
	\draw (8,8) node (b2) [black] {};
	\draw (b2) +(0,-10) node (b2prim) [black] {};
	\draw (b2) +(-10,0) node (b2prim2) [black] {};
	
	\draw (6,5) node (w1) [white] {};
	\draw (w1) +(0,10) node (w1prim) {};
	
	\draw (12,7) node (w2) [white] {};
	\draw (w2) +(-10,0) node (w2prim) [white] {};
	
	\draw[connection,pos=0.2] (b2) to node {\textcolor{black}{$4$}} node [swap] {} (w1prim);
	\draw[connection,pos=0.666] (b2prim) to node {\textcolor{black}{$4$}} node [swap] {} (w1);
	
	\draw[connection,pos=0.2] (b2) to node {\textcolor{black}{$6$}} node [swap] {} (w2);
	\draw[connection,pos=0.7] (b2prim2) to node {\textcolor{black}{$6$}} node [swap] {} (w2prim);
	
	\draw[connection] (b1) to node {\textcolor{black}{$2$}} node [swap] {} (b1-sw);
	
	\draw[connection] (b1) to node {\textcolor{black}{$3$}} node [swap] {} (b1-se);
	
	\draw[connection] (b1) to node {\textcolor{black}{$5$}} node [swap] {} (w1);
	
	\draw[connection] (w1) to node {} node [swap] {\textcolor{black}{$7$}} (b2);
	
	\draw[connection] (w2) to node {} node [swap] {\textcolor{black}{$1$}} (b1prim);
	\draw[connection] (w2prim) to node {} node [swap] {\textcolor{black}{$1$}} (b1);
	
	\draw[blue, line width=1mm,->] (w2prim)+(-1,-1) to (w2prim); 
	\end{scope}
	
	\draw[very thick,decoration={
		markings,
		mark=at position 0.666  with {\arrow{>}}},
	postaction={decorate}]  
	(0,0) -- (10,0);
	
	\draw[very thick,decoration={
		markings,
		mark=at position 0.666  with {\arrow{>}}},
	postaction={decorate}]  
	(0,10) -- (10,10);
	
	\draw[very thick,decoration={
		markings,
		mark=at position 0.666  with {\arrow{>>}}},
	postaction={decorate}]  
	(0,0) -- (0,10);
	
	\draw[very thick,decoration={
		markings,
		mark=at position 0.666  with {\arrow{>>}}},
	postaction={decorate}]  
	(10,0) -- (10,10);
	
	\end{tikzpicture}
	\caption{Example of an \emph{oriented map}. The map is drawn on a torus: the
		left side of the square should be glued to the right side, as well as bottom to
		top, as indicated by the arrows. The thick blue arrow indicates the root of the
		face. The orientation is indicated by the circular arrow. This map was created
		by glueing the edges of the oriented collection of polygons from
		\cref{fig:polygon-oriented}.} \label{fig:map-kerov}
\end{figure}

\subsection{Oriented maps}
\label{sec:oriented}

We say that a map is \emph{oriented} if the surface $\Sigma$ is orientable \emph{and}
on each connected component some orientation is selected, see \cref{fig:map-kerov}.

\emph{Summation over oriented maps with a specified face-type $\pi$}
should be understood as follows: we start by fixing an
appropriate collection $\mathcal{P}$ of polygons with labelled edges, bicolored
vertices, and with face-type $\pi$. \emph{Additionally, on each polygon we
	choose some orientation (``which direction of rotation should be understood as
	clockwise''),} see \cref{fig:polygon-oriented}. Then we consider \emph{only}
these perfect matchings on the set of all edges, which result with a glueing of
the polygons with the property that the originally selected orientations on the
polygons are consistent when one crosses an edge. We sum over the resulting
oriented map.

\subsection{Labels on oriented maps}
\label{sec:labels}

In the context of the \emph{oriented} maps it is more convenient to label only
some edges of the oriented polygons from~$\mathcal{P}$; more specifcially we
label only these edges which --- if one traverses the edges in the clockwise
cyclic order --- start with a white vertex, see \cref{fig:polygon-oriented}. 
The labelling can be encoded by a permutation $\pi$, the cycles of which correspond
to the cyclic order of the labels around the polygons (in the clockwise direction),
see \cref{fig:polygon-oriented}.  

With this convention, the orientations of the polygons are consistent in the
map obtained by glueing the edges of $\mathcal{P}$ if and only if in each pair
of glued edges exactly one carries a label, see \cref{fig:map-kerov}.

\medskip

The structure of such an oriented map can be uniquely recovered from pair of
permutations $\sigma_1, \sigma_2$, where the cycles of $\sigma_1$ encode the counterclockwise
cyclic order around the white vertices while the cycles of $\sigma_2$ encode the counterclockwise
cyclic order around the black vertices. For example, for the map from
\cref{fig:map-kerov} we have
\begin{align*}
\sigma_1 &= (1,6)(2)(3)(4,7,5),\\
\sigma_2 &= (1,2,3,5)(4,7,6).
\end{align*}

It is easy to see that the product
\[ \pi = \sigma_1 \sigma_2 \]
is the aforementioned permutation $\pi$ which gives the clockwise order of the
labels on the polygons from $\mathcal{P}$. In our example
\[ \pi=\sigma_1 \sigma_2 = (1,2,3,4,5,6,7), \]
see \cref{fig:polygon-oriented}.

\medskip

With this in mind, it is easy to see that, for a given integer partition $\pi$,
the \emph{summation over oriented maps with face-type $\pi$} is equivalent to
fixing some permutation $\pi\in\Sym{|\pi|}$ with the cycle decomposition given
by the partition $\pi$, and considering all solutions to the equation
\[ \big\{ (\sigma_1,\sigma_2) : \sigma_1,\sigma_2\in \Sym{|\pi|}, \quad \sigma_1 \sigma_2 = \pi \big\}.\]

\subsection{Proof of \cref{thm:maps}}
\label{sec:proof-oriented}

Our proof will be based on a simple double counting argument.

\begin{proof}[Proof of \cref{thm:maps}]
	In the light of the discussion from \cref{sec:labels}, the right-hand side of
	\eqref{eq:spin-stanley} can be interpreted as a sum over oriented maps with
	face-type $\pi$.
	
	\smallskip
	
	The oriented maps which we consider have labelled edges. We can remove these
	labels and still be able to recover them if we preserve the following
	information:
	\begin{itemize}
		\item we paint each face of the map with the colour of the corresponding
		polygon from $\mathcal{P}$;
		
		\item on each face of the map we decorate the white corner (\emph{``root''})
		which corresponds to the root of the corresponding polygon from~$\mathcal{P}$;
		
		\item on each connected component of the map we indicate the orientation.
	\end{itemize}
	
	\smallskip
	
	There are $2^{\ell(\pi)}$ ways to choose \emph{direction of the rotation} on
	each face a map (in the spirit of \cref{sec:non-oriented-revisited}). It follows
	from \eqref{eq:spin-stanley} that
	\[ \ChSpin_\pi(\xi) = \sum_M  
	\frac{1}{2^{c(M)}} \frac{1}{2^{\ell(\pi)}}\  (-1)^{|\pi|-|\mathcal{V}_\circ(M)|}\ N_{M}\big( D(\xi) \big),
	\]
	where the sum runs over oriented maps  with coloured faces, with each face having
	a decorated white corner, and with each face having some selected direction of
	rotation. Above, $c(M)$ denotes the number of connected components of $M$; with notations of 
	\eqref{eq:spin-stanley}
	we clearly have $c(M)=|\langle \sigma_1,\sigma_2\rangle|$.

	Let us remove the information about the orientation. In the light of
	\cref{sec:non-oriented-revisited}, the resulting object is a non-oriented map
	which happens to be orientable. Clearly, for each such a non-oriented but
	orientable map there are $2^{c(M)}$ choices for the orientation on each
	connected component of $M$. It follows that
	\[ \ChSpin_\pi(\xi) = \sum_M  
	\frac{1}{2^{\ell(\pi)}}\ (-1)^{|\pi|-|\mathcal{V}_\circ(M)|}\ N_{M}\big( D(\xi) \big),
	\]
	where the sum runs over non-oriented but orientable maps with face-type~$\pi$, as required.
\end{proof}

\section{Multirectangular coordinates, \\ spin Stanley polynomials}
	
\label{sec:multirectangular}

Our ultimate goal is to prove \cref{prop:isomorphism,thm:filtration} which
provide the link between (the filtration on) the linear Kerov--Olshanski algebra
$\KO$ and its spin counterpart $\Gamma$; we will do this in
\cref{sec:proof-of-isomorphism,sec:proof-of-filtration}.

We start in this section by preparing the tools: multirectangular coordinates
and Stanley polynomials.

\begin{figure}
	\begin{tikzpicture}	   
	   \clip (-0.3,-0.3) rectangle (9.5,6.5);
	   \fill[blue!10] (0,0) -- (7,0) -- (7,2) -- (4,2) -- (4,5) -- (0,5);
	   \draw[black!20] (0,0) grid (10,7);
	   \draw[ultra thick] (10,0) -- (0,0) -- (0,7);
	   \draw[thick] (7,0) -- (7,2) -- (0,2);
	   \draw[<->] (7.3,0) -- (7.3,2) node[midway,right] {$p_1$};
	   	   \draw[<->] (0,0.3) -- (7,0.3) node[midway,above] {$q_1$};
	   	   
	   	   \draw[thick] (4,2) -- (4,5) -- (0,5);
	   	   \draw[<->] (4.3,2) -- (4.3,5) node[midway,right] {$p_2$};
	   	   \draw[<->] (0,2.3) -- (4,2.3) node[midway,above] {$q_2$};
	   	   
	\end{tikzpicture}
	\caption{Multirectangular coordinates for partitions.}
	\label{fig:multirectangular}

\bigskip

	\begin{tikzpicture}	 
		   \clip (0.7,-0.3) rectangle (10.5,6.5);
		   \fill[blue!10] (4,4) -- (6,4) -- (6,2)-- (8,2) -- (8,0) -- (1,0) -- (1,1) -- (2,1) -- (2,2) -- (3,2) -- (3,3) -- (4,3) -- (4,4) -- (5,4);  
	\begin{scope}
	   \clip (11,0) -- (1,0) -- (1,1) -- (2,1) -- (2,2) -- (3,2) -- (3,3) -- (4,3) -- (4,4) -- (5,4) -- (5,5) -- (6,5) -- (6,6) -- (7,6) -- (7,7) -- (11,7);
	   	   \draw[black!20] (0,0) grid (11,11);
	\end{scope}
	
	\draw[ultra thick] (10,0) -- (1,0) -- (1,1) -- (2,1) -- (2,2) -- (3,2) -- (3,3) -- (4,3) -- (4,4) -- (5,4) -- (5,5) -- (6,5) -- (6,6) -- (7,6) -- (7,7);
	\draw[thick] (4,4) -- (6,4) -- (6,2) ;
	\draw[thick] (2,2) -- (8,2) -- (8,0);
		   \draw[<->] (8.3,0) -- (8.3,2) node[midway,right] {$\proP_1$};
	\draw[<->] (1,0.3) -- (8,0.3) node[midway,above] {$\proQ_1$};
	
	\draw[<->] (6.3,2) -- (6.3,4) node[midway,right] {$\proP_2$};
	\draw[<->] (3,2.3) -- (6,2.3) node[midway,above] {$\proQ_2$};
	
	\end{tikzpicture}
	\caption{Multirectangular coordinates for strict partitions.}
	\label{fig:multirectangular-strict}
\end{figure}

\subsection{Multirectangular coordinates} 

Following Stanley \cite{Stanley2003/04}, for tuples of integers
$P=(p_1,\dots,p_l)$, $Q=(q_1,\dots,q_l)$ such that $p_1,\dots,p_l\geq 0$ and
$q_1\geq \cdots\geq q_l\geq 0$ we consider the corresponding
\emph{multirectangular partition} $P\times Q$, cf.~\cref{fig:multirectangular}.

De Stavola \cite{DeStavolaThesis} adapted this notion to strict partitions: for
tuples of integers $\mathbf{P}=(\proP_1,\dots,\proP_l)$,
$\mathbf{Q}=(\proQ_1,\dots,\proQ_l)$ such that $\proP_1,\dots,\proP_l\geq 0$ and
$\proQ_2\leq \proQ_1-\proP_1$, $\proQ_3\leq \proQ_2-\proP_2$, \dots, $\proQ_l
\leq \proQ_{l-1}-\proP_{l-1}$, $0\leq \proQ_l-\proP_l$ he considered a
\emph{multirectangular strict partition} $\mathbf{P}\stimes\mathbf{Q}$,
cf.~\cref{fig:multirectangular-strict}; note that in the original work of De
Stavola both the multirectangular partition as well as the multirectangular
strict partition were denoted by the same symbol $\times$.
	
\medskip

The following result gives a link between the shifted multirectangular
coordinates of a given strict partition $\xi\in\SP$ and the multirectangular
coordinates of its double $D(\xi)$.

\begin{lemma}
	\label{lem:multi-strict-vs-usual} 
	
For each $l\geq 1$ there exist polynomials $p_1,\dots,p_l,q_1,\dots,q_l\in \Z[
\proP_1,\dots,\proP_l, \proQ_1,\dots,\proQ_l]$ of degree $1$ with the property
that
	\[ D(\mathbf{P}\stimes\mathbf{Q}) = P \times Q.\]
\end{lemma}
\begin{proof}
The desired polynomials are given by
\begin{equation*}
\label{eq:double-multirectangular}
\left\{
\begin{aligned}
	p_1 &= \proP_1, & q_1 &= \proQ_1+1,\\
	p_2 &= \proP_2, & q_2 &= \proP_1+\proQ_2+1, \\
	    &\vdots     &     & \vdots \\
	p_l &= \proP_l, & q_l &= \proP_1+\cdots+\proP_{l-1}+\proQ_l+1,\\
	p_{l+1} &=  \proQ_l - \proP_l, & q_{l+1} &= \proP_1+\cdots+\proP_l, \\
	p_{l+2} &=  \proQ_{l-1} - (\proQ_l + \proP_{l-1}), & 
	q_{l+2} &= \proP_1+\cdots+\proP_{l-1}, & \\
	        & \vdots &    & \vdots \\
	p_{2l}  &= \proQ_1- (\proQ_2 + \proP_1), & q_{2l}&= \proP_1. 
\end{aligned}
\right.	
\end{equation*}
\end{proof}

\subsection{Stanley polynomials}

Multirectangular coordinates $P,Q$ provide a convenient way of parametrizing
partitions; the fact that a partition can be represented in several ways with
such coordinates is not relevant.

In particular, for a function $F\in\KO$ it is convenient to consider a map
\begin{equation} 
\label{eq:Stanley}
(P,Q) \mapsto F(P\times Q)
\end{equation} 
which --- as it turns out --- can be identified with a unique polynomial in the
multirectangular coordinates. This polynomial, which is referred to as
\emph{Stanley polynomial} \cite{Stanley2003/04}, is a convenient tool for
studying some enumerative and asymptotic problems of the representation theory
of the symmetric groups \cite{DolegaFeraySniady2008}.

The explicit form of this Stanley polynomial $\Ch_\pi(P\times Q)$ in the special
case when $F=\Ch_\pi$ is the normalized character was conjectured by Stanley
\cite{Stanley-preprint2006} and proved by F\'eray \cite{Feray2010}, see also
\cite{FeraySniady2011a}; we will refer to this result as \emph{Stanley character
	formula} (it is just a reformuation of \cref{thm:stanley-linear}).

\medskip

De Stavola \cite{DeStavolaThesis} initiated investigation of the analogous
coordinates for a function $F\in\Gamma$ and he asked for a closed formula for
$\ChSpin_\pi(\mathbf{P}\stimes\mathbf{Q})$.

\subsection{Closed formula for spin Stanley polynomial}
\label{sec:Stanley-strict} 

Multirectangular coordinates are very convenient in the context of Stanley character formula. 
More specifically, it is easy to see that
\begin{equation} 
\label{eq:explicit-stanley}
N_{\sigma_1,\sigma_2}(P\times Q) = 
\sum_{\kappa_2} 
\prod_{c_1\in C(\sigma_1)} q_{\kappa_1(c_1)}  
\prod_{c_2\in C(\sigma_2)} p_{\kappa_2(c_2)},  
\end{equation}
where the sum runs over all functions
$\kappa_2\colon C(\sigma_2) \to [l]$
and the function $\kappa_1\colon C(\sigma_1)\to[l]$ is defined by
\[ \kappa_1(c_1)= \max\big\{ \kappa_2(c_2) : 
c_2\in C(\sigma_2)\text{ and } c_1\cap c_2\neq \emptyset \big\}\quad \text{for } c_1\in C(\sigma_1).\]
The explicit form of the Stanley polynomial
$\ChSpin_\pi(\mathbf{P}\stimes\mathbf{Q})$ for $\pi\in\OP$ can be found by
combining \cref{thm:spin-Stanley} applied to $\xi=\mathbf{P}\stimes\mathbf{Q}$ and 
$P\times Q=D(\mathbf{P}\stimes\mathbf{Q})$ with
\cref{lem:multi-strict-vs-usual}. In order to prove the uniqueness of this
polynomial one can adapt the corresponding part of the proof of \cite[Lemma
2.4]{DolegaFeraySniady2014}. 

\section{Asymptotics of spin Stanley polynomials}
\label{sec:asymptotic-stanley}

Recall that our ultimate goal is to prove \cref{prop:isomorphism,thm:filtration} which
provide the link between (the filtration on) the linear Kerov--Olshanski algebra
$\KO$ and its spin counterpart $\Gamma$.

As an intermediate step we present in this section the link between the
filtration on the (linear or spin) Kerov--Olshanski algebra and the degrees of
Stanley polynomials
(\cref{prop:filtration-and-degree-linear,prop:filtration-and-degree-spin}) which
might be of independent interest.

\subsection{Filtration vs Stanley polynomials: the linear case}

\begin{lemma}
	\label{lem:top-degree-spin-Stanley-linear}
	For arbitrary integer partition $\pi$ the corresponding linear Stanley
	polynomial $\Ch_\pi(P\times Q)$ is of degree
	$|\pi|+\ell(\pi)$. 
	
	Its homogeneous top-degree part is given by
	\begin{equation} 
	\label{eq:homogeneous-top-degree-stanley-linear}
	\left(\Ch_\pi\right)^{\ttop}(P\times Q):=
	\sum_{\substack{\sigma_1,\sigma_2\in\Sym{k} \\ \sigma_1 \sigma_2=\pi \\ |C(\sigma_1)|+|C(\sigma_2)|=|\pi|+\ell(\pi)}} 
	 (-1)^{\sigma_1}\ N_{\sigma_1,\sigma_2}(P\times Q).  
	\end{equation}
\end{lemma}
\begin{proof}
	Our strategy is to investigate the summands on the right-hand side of the
linear Stanley character formula \eqref{eq:stanley-linear} . Note that
\eqref{eq:explicit-stanley} implies that
	\begin{equation} 
	\label{eq:embedding-homogeneous}
	N_{\sigma_1,\sigma_2}\big( P \times Q \big) \in \Z[p_1,\dots,p_l,q_1,\dots,q_l] 
	\end{equation}
	is a homogeneous polynomial of degree $|C(\sigma_1)|+|C(\sigma_2)|$.

	For a permutation
	$\pi\in\Sym{k}$ we denote by $\|\pi\|:=k - |C(\pi)|$ the minimal number of
	factors necessary to write $\pi$ as a product of transpositions. Triangle
	inequality for $\pi=\sigma_1 \sigma_2$ implies that
	\begin{equation}
	\label{eq:triangle}
	|C(\sigma_1)|+|C(\sigma_2)| = 2k - \big( \|\sigma_1| + \|\sigma_2\| \big) \leq 2k - \| \pi \| = |\pi| + \ell(\pi).
	\end{equation}
	It follows therefore that 
	\[ \Ch_\pi(P\times Q) \]
	is indeed a polynomial of degree bounded from above by $|\pi| + \ell(\pi)$, as required.
	
	In order to find the homogeneous top-degree part of this polynomial it is enough
	to restrict the summation in \eqref{eq:spin-stanley} to the pairs
	$\sigma_1,\sigma_2$ for which \eqref{eq:triangle} becomes equality. 
	\end{proof}

\begin{proposition}
	\label{prop:filtration-and-degree-linear}
	Let $F\in \KO$ and $k\geq 0$ be an integer.
	
	Then $F\in \F_k$ if and only if for each integer $l\geq 1$ 
	\begin{equation}
	\label{eq:my-stanley}
	F(P\times Q) \in \C[p_1,\dots,p_l,q_1,\dots,q_l] 
	\end{equation} 
	is a polynomial of degree at most $k$.
\end{proposition}
\begin{proof}
	
	The vector space $\F_k$ is a linear span of certain normalized characters; for
	this reason it is enough to consider $F=\Ch_\pi\in\F_k$ with $\pi$ as in
	\eqref{eq:filtration}. 
Application of \cref{lem:top-degree-spin-Stanley-linear}  concludes the proof
of the implication in one direction.

	\medskip
	
	We shall prove the opposite implication.
	Let $m\geq 0$ be the minimal integer with the
	property that $F\in\F_m$; our goal is to show that $m\leq k$. Suppose that this
	is not the case and $m>k$. 
	We consider the map
	\begin{multline}
	\label{eq:homogeneous}
	\Psi_m \colon \F_m \ni G \mapsto \\ \left. \Big( [\text{homogeneous part of degree $m$}]\
	G(P\times Q) \Big) \right|_{p_1=\cdots=p_l=1}
	\end{multline}
	which selects the top-degree homogeneous part of Stanley's polynomial and
\emph{afterwards} substitutes $p_1=\cdots=p_l=1$. We will show that $\ker
\Psi_m = \F_{m-1}$, provided $l\geq m$; this would conclude the proof because
$m>k$ implies that $F\in \ker \Psi_m = \F_{m-1}$ which contradicts the 
minimality of $m$.
	
	\medskip
	
	Clearly the quotient space $\F_{m}/\F_{m-1}$ is a linear span of 
	\begin{equation}
	\label{eq:quotient}
	\big(\Ch_\pi + \F_{m-1} : |\pi|+\ell(\pi)=m \big).
	\end{equation}
	The map $\Psi_m$ maps the family \eqref{eq:quotient} to 
	\begin{equation}
	\label{eq:quotient2} 
	\big(\Psi_m (\Ch_\pi) : |\pi|+\ell(\pi)=m \big); 
	\end{equation}
	for our purposes of proving $\ker \Psi_m=\F_{m-1}$ it is enough to show that
\eqref{eq:quotient2} is a family of linearly independent vectors. We shall do
it in the following.
	
	\smallskip
	
	In order to find the top-degree part of the polynomial
	\begin{equation}
	\label{eq:substitute} 
	\Psi_m (\Ch_\pi)\in\Z[q_1,\dots,q_l]
	\end{equation}
	we combine \eqref{eq:homogeneous-top-degree-stanley-linear} with
\eqref{eq:explicit-stanley}. This top-degree part corresponds to these summands
in \eqref{eq:homogeneous-top-degree-stanley-linear} for which $|C(\sigma_1)|$
takes the maximal possible value. This is clearly the summand for
$\sigma_1=\operatorname{id}\in\Sym{|\pi|}$ and $\sigma_2=\pi\in\Sym{|\pi|}$. In
this way we proved that \eqref{eq:substitute} is a polynomial of degree $|\pi|$
and its homogeneous part of this maximal degree is equal to the power-sum
symmetric function
	 \begin{equation} 
	\label{eq:power-sum} p_\pi(q_1,\dots,q_l).
	\end{equation}

Suppose that 
\begin{equation}
\label{eq:combination} \sum_{|\pi|+\ell(\pi)=m} a_\pi\ \Psi_m (\Ch_\pi) = 0 
\end{equation}
for some coefficients $a_\pi\in\C$ which are not all equal to
zero and denote $n:=\max \{ |\pi| : a_\pi\neq 0 \}$. We consider the homogeneous
part of the polynomial on the left-hand side of \eqref{eq:combination} of degree
$n$; by maximality of $n$, only partitions $\pi$ with $|\pi|=n$ contribute, and
each such a partition contributes with the power-sum symmetric function
\eqref{eq:power-sum}. Since power-sum symmetric polynomials in $l$ variables of
degree $n\leq m\leq l$ are linearly independent, it follows that $a_\pi=0$ for
$|\pi|=n$ which contradicts the definition of $n$.
This concludes the proof of the equality $\ker \Psi_m=\F_{m-1}$.
\end{proof}

This result has a spin counterpart (\cref{prop:filtration-and-degree-spin}),
however adapting the above proof to the spin setup requires
some preparations which we present below.

\subsection{Overlap double of a strict partition}
\label{sec:symmetric-double}

\begin{figure}

	\begin{tikzpicture}[xscale=0.6,yscale=0.6]
	\begin{scope}
	\clip (0,0) -- (0,1) -- (1,1) -- (1,2) -- (2,2) -- (2,3) -- (3,3) -- (4,3) -- (4,2) -- (6,2) -- (6,1) -- (6,0);
	\draw[gray] (0,0) grid (8,3);
	\end{scope}
	\begin{scope}[rotate=90,yscale=-1]
	\clip (0,0) -- (0,1) -- (1,1) -- (1,2) -- (2,2) -- (2,3) -- (3,3) -- (4,3) -- (4,2) -- (6,2) -- (6,1) -- (6,0);
	\draw[gray] (0,0) grid (8,3);
	\end{scope}    
	\begin{scope}
	\draw[ultra thick] (3,3) -- (4,3) -- (4,2) -- (6,2) -- (6,1) -- (6,0) -- (0,0);
	\draw[line width=0.2cm, opacity=0.4,blue,line cap=round] (0.5,0.5) -- (5.5,0.5);
	\draw[line width=0.2cm, opacity=0.4,blue,line cap=round] (1.5,1.5) -- (5.5,1.5);
	\draw[line width=0.2cm, opacity=0.4,blue,line cap=round] (2.5,2.5) -- (3.5,2.5);	 
	\end{scope}
	\begin{scope}[rotate=90,yscale=-1]
	\draw[ultra thick]  (3,3) -- (4,3) -- (4,2) -- (6,2) -- (6,1) -- (6,0) -- (0,0);
	\draw[line width=0.2cm, opacity=0.4,OliveGreen,line cap=round] (0.5,0.5) -- (5.5,0.5);
	\draw[line width=0.2cm, opacity=0.4,OliveGreen,line cap=round] (1.5,1.5) -- (5.5,1.5);
	\draw[line width=0.2cm, opacity=0.4,OliveGreen,line cap=round] (2.5,2.5) -- (3.5,2.5);	
	\end{scope}    
	\end{tikzpicture}
	
	\caption{The overlap double $\Dover(\xi)=(6,6,4,3,2,2)$ of the strict
	partition $\xi=(6,5,2)$ from \cref{fig:strict}.} 
\label{fig:doubleS}
\end{figure}
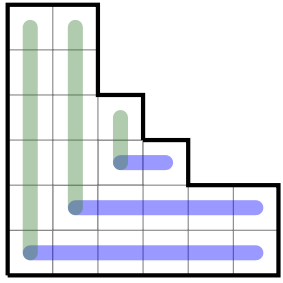

For a strict partition $\xi\in\SP$ we consider its \emph{overlap double}
$\Dover(\xi)\in \Part$ which is obtained by combining the shifted Young diagram
$\xi$ and its transpose so that they overlap along the diagonal boxes,
cf.~\cref{fig:doubleS}.

Let $p_1,\dots,p_l,q_1,\dots,q_l\in \Z[ \proP_1,\dots,\proP_l,
\proQ_1,\dots,\proQ_l]$ be the homogeneous part of degree $1$ of the polynomials
\eqref{eq:double-multirectangular}; in other words
\begin{equation}
\label{eq:double-multirectangular-H}
\left\{
\begin{aligned}
p_1 &= \proP_1, & q_1 &= \proQ_1,\\
p_2 &= \proP_2, & q_2 &= \proP_1+\proQ_2, \\
&\vdots     &     & \vdots \\
p_l &= \proP_l, & q_l &= \proP_1+\cdots+\proP_{l-1}+\proQ_l,\\
p_{l+1} &=  \proQ_l - \proP_l, & q_{l+1} &= \proP_1+\cdots+\proP_l, \\
p_{l+2} &=  \proQ_{l-1} - (\proQ_l + \proP_{l-1}), & 
q_{l+2} &= \proP_1+\cdots+\proP_{l-1}, & \\
& \vdots &    & \vdots \\
p_{2l}  &= \proQ_1- (\proQ_2 + \proP_1), & q_{2l}&= \proP_1. 
\end{aligned}
\right.	
\end{equation}

It is easy to check that with this choice
\begin{equation} 
\label{eq:symmetric-double-multi}
\Dover(\mathbf{P}\stimes\mathbf{Q}) = P \times Q.
\end{equation}

\subsection{Top-degree part of spin Stanley polynomials} 

\cref{lem:top-degree-spin-Stanley-linear} has the following spin counterpart.

\begin{lemma}
\label{lem:top-degree-spin-Stanley} 

For arbitrary $\pi\in\OP_k$ the corresponding spin Stanley polynomial
$\ChSpin_\pi(\mathbf{P}\stimes\mathbf{Q})$ is of degree $|\pi|+\ell(\pi)$. Its
homogeneous top-degree part is given by
\begin{multline} 
\label{eq:homogeneous-top-degree-stanley}
\left(\ChSpin_\pi\right)^{\ttop}(\mathbf{P}\stimes\mathbf{Q}):=
\\
	\sum_{\substack{\sigma_1,\sigma_2\in\Sym{k} \\ \sigma_1 \sigma_2=\pi \\ |C(\sigma_1)|+|C(\sigma_2)|=|\pi|+\ell(\pi)}} 
	\frac{1}{2^{|\sigma_1\vee \sigma_2|}}\ (-1)^{\sigma_1}\ N_{\sigma_1,\sigma_2}\big( \Dover(\mathbf{P} \stimes\mathbf{Q}) \big).  
\end{multline}
\end{lemma}
\begin{proof}
We revisit the proof of \cref{lem:top-degree-spin-Stanley-linear} but this time
we consider the spin Stanley formula \eqref{eq:spin-stanley}.
\cref{lem:multi-strict-vs-usual} implies therefore that
\begin{equation} 
\label{eq:embedding-inhomogeneous}
N_{\sigma_1,\sigma_2}\big( D(\mathbf{P} \stimes\mathbf{Q}) \big) \in \Z[\proP_1,\dots,\proP_l,\proQ_1,\dots,\proQ_l] 
\end{equation}
is an \emph{inhomogeneous} polynomial of the degree $|C(\sigma_1)|+|C(\sigma_2)|$, obtained from
\eqref{eq:embedding-homogeneous} by the substitution
\eqref{eq:double-multirectangular}. 

The homogeneous part of \eqref{eq:embedding-inhomogeneous} of this top degree
can be therefore obtained from \eqref{eq:embedding-homogeneous} by the
\emph{homogeneous} substitution \eqref{eq:double-multirectangular-H}. Equation
\eqref{eq:symmetric-double-multi} implies therefore that this homogeneous part
is equal to
\begin{equation} 
\label{eq:embedding-homogeneous-again}
 N_{\sigma_1,\sigma_2}\big( \Dover(\mathbf{P} \stimes\mathbf{Q}) \big) \in \Z[\proP_1,\dots,\proP_l,\proQ_1,\dots,\proQ_l]. 
\end{equation}

\medskip

Inequality \eqref{eq:triangle} implies therefore that 
\[ \ChSpin_\pi(\mathbf{P}\stimes\mathbf{Q}) \]
is indeed a polynomial of degree bounded from above by $|\pi| + \ell(\pi)$, as required.

In order to find the homogeneous top-degree part of this polynomial it is enough
to: (a) restrict the summation in \eqref{eq:spin-stanley} to pairs
$\sigma_1,\sigma_2$ for which \eqref{eq:triangle} becomes equality and then (b)
to consider the homogeneous part \eqref{eq:embedding-homogeneous-again} of the
surviving summands. This concludes the proof.
\end{proof}

\subsection{Filtration vs Stanley polynomials: the spin case}

\begin{proposition}
	\label{prop:filtration-and-degree-spin}
	Let $F\in \Gamma$ and $k\geq 0$ be an integer.
	
	Then $F\in \G_k$ if and only if for each integer $l\geq 1$ 
	\[ F(\mathbf{P}\stimes\mathbf{Q}) \in \C[\proP_1,\dots,\proP_l,\proQ_1,\dots,\proQ_l] \]
	is a polynomial of degree at most $k$.
\end{proposition}

\begin{proof}
We revisit the proof of \cref{prop:filtration-and-degree-linear} and review the necessary changes. 
	
We replace each occurrence of $\F_k$ by $\G_k$ and each linear character
$\Ch_\pi$ with $\pi\in\Part$ by its spin counterpart $\ChSpin_\pi$ with
$\pi\in\OP$. Reference to \cref{lem:top-degree-spin-Stanley-linear} should be
replaced by \cref{lem:top-degree-spin-Stanley}.

Instead of $\Psi_m$ we consider the map
\begin{multline}
\label{eq:homogeneous-spin}
\Psi^{\text{spin}}_m \colon \G_m \ni G \mapsto \\ \left. \Big( [\text{homogeneous part of degree $m$}]\
G(\mathbf{P}\stimes\mathbf{Q}) \Big) \right|_{\proP_1=\cdots=\proP_l=1}
\end{multline}
which selects the top-degree homogeneous part of spin Stanley's polynomial and
\emph{afterwards} substitutes $\proP_1=\cdots=\proP_l=1$.

\medskip

The only part which requires some changes in the spin context is the proof that
if $m=|\pi|+\ell(\pi)$ then
$\Psi^{\text{spin}}_m(\ChSpin_\pi)\in\Q[\proQ_1,\dots,\proQ_l]$ is a polynomial
of degree $|\pi|$ and its homogeneous part of this maximal degree is equal to
the power-sum symmetric function $p_\pi(\proQ_1,\dots,\proQ_l)$.
We present this
proof below.

\smallskip

Let us fix some $\sigma_1,\sigma_2\in\Sym{|\pi|}$ which contribute to
\eqref{eq:homogeneous-top-degree-stanley}. The colorings which contribute to
$N_{\sigma_1,\sigma_2}(\lambda)$ with 
\[\lambda=\Dover\big( (1,\dots,1) \stimes
\mathbf{Q} \big)= \Dover(\proQ_1,\dots,\proQ_l)\] 
can be enumerated by the following algorithm. Firstly, we select some set of
cycles of $\sigma_1$ and some set of cycles of $\sigma_2$ (we will call the
cycles which belong to them \emph{special cycles}). Secondly, we assign to these
special cycles the values of $f_1$ and $f_2$ from the set $[l]$ in an arbitrary
way. 
The above two steps do not depend on the variables $\proQ_1,\dots,\proQ_l$, but they do depend on $l$, the number of the variables.

In the third step we associate to all non-special cycles the values of $f_1$ and
$f_2$ from the set $\{l+1,l+2,\dots\}$ in such a way that $f=(f_1,f_2)$ is a
colouring of $(\sigma_1,\sigma_2)$ which is compatible with $\lambda$. 
Without loss of generality we may assume that
whenever two cycles $c_1\in C(\sigma_1)$, $c_2\in
C(\sigma_2)$ intersect, at least one of them is special; otherwise the column
$f_1(c_1)$ and the row $f_2(c_2)$ would intersect outside of the Young
diagram $\lambda$ and there would be no such colourings which are compatible with $\lambda$. 
For each $i\in [l]$ in $i$-th row of $\lambda$ there are
$\proQ_{i}+i-1$ 
boxes and $l$ of them are forbidden by non-speciality requirement; it follows that
the number of choices in this third step is equal to
	\begin{multline}
	\label{eq:number-of-choices} 
	\prod_{\substack{c_1 \in C(\sigma_1), \\ \text{$c_1$ is not special} }} \left[ \proQ_{g_1(c_1)} + g_1(c_1)-1 -l  \right] 
	\times \\
	\prod_{\substack{c_2 \in C(\sigma_2), \\ \text{$c_2$ is not special} }} 
	\left[ \proQ_{g_2(c_2)} + g_2(c_2)- 1- l \right] ,
	\end{multline}
		 where for $c_1\in C(\sigma_1)$, $c_2\in C(\sigma_2)$ which are not special we define
	\begin{align*} 
	g_1(c_1) &= \max \big\{ f_2(c_2) : c_2\in C(\sigma_2), \text{$c_2$ is special} \\ & \qquad\qquad\qquad\qquad \text{and the cycles $c_1$ and $c_2$ intersect} \big\},\\
	g_2(c_2) &= \max \big\{ f_1(c_1) : c_1\in C(\sigma_1), \text{$c_1$ is special} \\ & \qquad\qquad\qquad\qquad \text{and the cycles $c_1$ and $c_2$ intersect} \big\}.
	\end{align*} 

\smallskip

The exact form of \eqref{eq:number-of-choices} is not really important; we will
use only the observation that it is a polynomial in $\proQ_1,\dots,\proQ_l$ of
degree equal to the number of non-special cycles. Since we are interested in the
top-degree part of the polynomial
$\Psi^{\text{spin}}_m(\ChSpin_\pi) \in\Q[\proQ_1,\dots,\proQ_l]$, hence our goal is
to maximize the number of such non-special cycles.

\smallskip

In the following it will be sometimes convenient to identify a given cycle $c\in
C(\sigma_1)\sqcup C(\sigma_2)$ with its support $c\subseteq [|\pi|]$ which is a non-empty subset of $[|\pi|]=\{1,\dots,|\pi|\}$. 
We denote by
$N\subseteq C(\sigma_1)\sqcup C(\sigma_2)$ the set of all non-special cycles. We
consider an arbitrary map $s\colon N\to[|\pi|]$ which to any non-special cycle
associates one of its elements, i.e.~$f(c)\in c$ for $c\in C(\sigma_1)\sqcup
C(\sigma_2)$. The map $s$ is injective; otherwise this would contradict the
assumption that any two non-special cycles are disjoint. This has twofold
consequences. 

Firstly, the number of non-special cycles is bounded from above by
$|\pi|$ so the degree of the polynomial
$N_{\sigma_1,\sigma_2}(\lambda)\in\Z[\proQ_1,\dots,\proQ_l]$ is bounded from
above by $|\pi|$. 

Secondly, if this bound on the number of non-special cycles is saturated, the
map $s$ is a bijection. In this case for each $x\in[|\pi|]$ exactly one of the
following two possibilities holds true: either 
\begin{itemize}
	\item $(x)$ is a non-special cycle
of $\sigma_1$ and $x$ belongs to some special cycle of $\sigma_2$, or 
\item $(x)$ is
a non-special cycle of $\sigma_2$ and $x$ belongs to some special cycle of
$\sigma_1$.
\end{itemize}

The above discussion implies that the product $\sigma_1 \sigma_2\in\Sym{|\pi|}$
can be calculated in a very simple way. Indeed, let us consider some special
cycle $c$; let us say that it is a cycle of $\sigma_1$. Since $\sigma_2$
restricted to the support of $c$ is equal to the identity, the restriction of
the product $\sigma_1 \sigma_2$ to the support of $c$ coincides with the cycle $c$.
An analogous result remains true in the case when $c$ is a special cycle of
$\sigma_2$. In this way we proved that the product $\sigma_1\sigma_2$ coincides
with the collection of the special cycles.

\smallskip

In this way we proved that if the number of non-special cycles achieves its
maximal value $|\pi|$, each cycle of $\pi=\sigma_1\sigma_2\in\Sym{|\pi|}$ is either a special cycle of
$\sigma_1$ or a special cycle of $\sigma_2$ and there are no other special
cycles.  

For example, if $\pi=(\pi_1)\in\OP$ is an odd partition which consists of only one part and 
\[ \pi=(1,2,\dots,\pi_1)\in\Sym{\pi_1}\]
is the corresponding permutation, the maximal number of non-special cycles is obtained for the following two choices:
\begin{align*}
\sigma_1&=\id=\underbrace{(1)(2)\cdots(\pi_1)}_{\text{non-special cycles}}, &  \sigma_2&=\underbrace{(1,2,\dots,\pi_1)}_{\text{special cycle}}, \\
\intertext{and}
\sigma_1&=\underbrace{(1,2,\dots,\pi_1)}_{\text{special cycle}}, &
\sigma_2&= 
\id=\underbrace{(1)(2)\cdots(\pi_1)}_{\text{non-special cycles}}.
\end{align*}
If $j$ denotes the value of the function $f_i$, $i\in\{1,2\}$, on the unique
special cycle, the homogeneous top-degree part of the corresponding summand
\[ \frac{1}{2^{|\sigma_1\vee \sigma_2|}}\ (-1)^{\sigma_1}\ N_{\sigma_1,\sigma_2}(\lambda) \]
on the right-hand side of \eqref{eq:homogeneous-top-degree-stanley} is equal to
\[ \frac{1}{2} \sum_j \proQ_j^{\pi_1} \]
because $\pi_1$ is an odd integer.

In general, for each cycle of $\pi$ there are two choices analogous to the ones above.
It follows that in total there are $2^{\ell(\pi)}$ choices for $\sigma_1$ and
$\sigma_2$ and which of the cycles are special.
For each such a choice the total
contribution of \eqref{eq:number-of-choices} is a polynomial of degree $|\pi|$,
with the homogeneous part equal to $\frac{1}{2^{\ell(\pi)}}
p_{\pi}(\proQ_1,\dots,\proQ_l)$, a multiple  of the power-sum symmetric
polynomial. This concludes the proof. 
\end{proof}

\section{Proof of \cref{prop:isomorphism,thm:filtration}}
\label{sec:proof-of-isomorphism}
\label{sec:proof-of-filtration}

\begin{proof}[Proof of \cref{prop:isomorphism}]
\

	\emph{Proof that $\double(\KO)\subseteq\Gamma$.}
	The algebra $\KO$ has an algebraic basis given by the functions $S_2,S_3,\dots$ with
	\[ S_k(\lambda) = (k-1) \iint_{(x,y)\in\lambda} (x-y)^{k-2} \dif x \ \dif y,
	 \]
	see \cite{DolegaFeraySniady2008}.
	
	The $i$-th row of a shifted Young diagram $\xi$ corresponds to two rectangles in the double $D(\xi)$. It follows that	
	\begin{multline*} 
	(\double T_k)(\xi) = \\
	\frac{1}{k} \sum_{i}  \left( - \xi_i^k + (\xi_i+1)^k + 0^k - 1^k \right) +  \left( -(1-\xi_i)^k + (-\xi_i)^k +1^k - 0^k \right).
	\end{multline*}
	It is easy to check that the polynomial in the variable $\xi_i$ which appears
in the above formula is an odd polynomial; it follows therefore that $\double
T_k\in \C[p_1,p_3,p_5,\dots]$ is a symmetric function which is supersymmetric.
By a result of Ivanov \cite[Section 6]{Ivanov2004} it follows that $\double T_k\in \Gamma$.

	\medskip
	
	\emph{Proof that $\double$ is surjective.} We will
	prove a stronger result $\G_k \subseteq \double(\F_k)$ below, in the proof of
	\cref{thm:filtration}.
\end{proof}

\begin{proof}[Proof of \cref{thm:filtration}]
	
\
	
\emph{Proof that the family $(\G_k)$ is a filtration on the algebra $\Gamma$.}
Suppose that $F_1\in \G_{k_1}$ and $F_2\in \G_{k_2}$. We apply
\cref{prop:filtration-and-degree-spin}; for $i\in\{1,2\}$ it follows that
$F_i(\mathbf{P}\stimes\mathbf{Q})\in\C[\proP_1,\dots,\proP_l,\proQ_1,\dots,\proQ_l]$ 
is a polynomial of degree at most $k_i$ hence
the product $F_1(\mathbf{P}\stimes\mathbf{Q}) F_2(\mathbf{P}\stimes\mathbf{Q})$
is a polynomial of degree at most $k_1+k_2$. We apply
\cref{prop:filtration-and-degree-spin} again; it follows that $F_1 F_2\in
\G_{k_1+k_2}$. This implies that $(\G_k)$ is indeed a filtration.
	
\medskip
	
\emph{Proof of the inclusion $\double(\F_k) \subseteq \G_k$.} Suppose $F\in\F_k$. By
\cref{prop:filtration-and-degree-linear} and \cref{lem:multi-strict-vs-usual} it
follows that $\double(F)(\mathbf{P}\stimes\mathbf{Q})$ is a polynomial of degree
at most $k$. We apply \cref{prop:filtration-and-degree-spin}; it follows that
$\double(F)\in \G_k$, as required.

\medskip

\emph{Proof of the inclusion $\G_k \subseteq \double(\F_k)$.} Let us fix an integer $k\geq
0$. Let $\pi$ be an arbitrary integer partition such that $|\pi|+\ell(\pi)\leq
k$. We define
\[ x_\pi:= - \sum_{\partition} \left(-\frac{1}{2}\right)^{|\partition|}\ (2|\partition|-3)!!\
\prod_{b\in \partition} \Ch_{(\class_i:i \in b)} \in \F_k,
\]
where the sum runs over all set-partitions of the set $[\ell(\class)]$.
\cref{theo:spin-in-linear} implies that $\double\left(x_\pi\right)=\ChSpin_\pi$;
in this way we proved that $\ChSpin_\pi\in \double(\F_k)$. From the definition
\eqref{eq:filtration2} of $\G_k$ it follows that $\G_k \subseteq \double(\F_k)$,
as required.
\end{proof}

\section*{Acknowledgements}

Research of SM was supported by JSPS KAKENHI Grant Number 17K05281.
Research of P\'S was supported by \emph{Narodowe Centrum Nauki}, grant number
2017/26/A/ST1/00189.

We thank Valentin F\'eray and Maciej Do\l\k{e}ga for several inspiring discussions. 

\biblio

\end{document}